\documentclass[11pt]{amsart}

\usepackage{graphicx,color,latexsym}
\usepackage{amssymb}
\usepackage{pb-diagram}

\newtheorem{theorem}{Theorem}[section]
\newtheorem{prop}[theorem]{Proposition}
\newtheorem{lemma}[theorem]{Lemma}

\theoremstyle{definition}
\newtheorem{remark}[theorem]{Remark}
\newtheorem{question}[theorem]{Question}
\newtheorem{definition}[theorem]{Definition}

\newtheorem{example}[theorem]{Example}

\begin{document}

\title{Relative Ruan and Gromov-Taubes Invariants of Symplectic 4-Manifolds}

\author{Josef G. Dorfmeister}
\address{Institut f\"ur Differentialgeometrie\\  Leibniz Universit\"at Hannover\\ 30167 Hannover, Germany\\
j.dorfmeister@math.uni-hannover.de}

\author{Tian-Jun Li}
\address{School  of Mathematics\\  University of Minnesota\\ Minneapolis, MN 55455\\
tjli@math.umn.edu}

\date{\today}
\begin{abstract}
We define relative Ruan invariants that count embedded connected symplectic submanifolds which contact a fixed symplectic hypersurface $V$ in a symplectic 4-manifold $(X,\omega)$ at prescribed points with prescribed contact orders (in addition to insertions on $X\backslash V$) for stable $V$.  We obtain invariants of the deformation class of $(X,V,\omega)$.  Two large issues must be tackled to define such invariants: (1) Curves lying in the hypersurface $V$ and (2) genericity results for almost complex structures constrained to make $V$ pseudo-holomorphic (or almost complex).  Moreover, these invariants are refined to take into account rim tori decompositions.  In the latter part of the paper, we extend the definition to disconnected submanifolds and construct relative Gromov-Taubes invariants.

\end{abstract}

\maketitle

\tableofcontents

\section{Introduction}

Gromov's paper \cite{Gr1} initiated the intense study of pseudoholomorphic curves in symplectic manifolds as a means to construct symplectic invariants.  This has led to a wide range of invariants for symplectic manifolds which "count" such curves.    After the development of the symplectic sum and symplectic cut (\cite{Le}, \cite{Go2}, \cite{MW}), research turned also to developing invariants relative to a fixed symplectic hypersurface $V$, i.e. a symplectic submanifold (or a smooth divisor) of codimension 2.  A variety of relative invariants have been developed which consider curves in $X$ which contact $V$ in a specified manner (\cite{IP4}, \cite{LR}, \cite{Li1}, \cite{LZ}).  In particular Gromov-Witten theory has seen great advances in the last few years:  Tools such as the decomposition formulas of \cite{LR}, \cite{IP5} and \cite{Li2} have been developed, correspondences between different invariants have been found in \cite{MP2} and \cite{HLR} and the invariants have been explicitly calculated for a range of manifolds, see any of the sources cited above.  Gromov-Witten invariants are  based on moduli spaces of maps from (nodal) Riemann surfaces to a symplectic manifold $(X,\omega)$.  Much of the work involved in defining these invariants has to do with the intractable nature of these moduli spaces.

A more natural object of study in the symplectic category are embedded symplectic submanifolds.  This topic was initially studied by Ruan \cite{R}, who initiated the development of invariants which count embedded submanifolds in four-manifolds.  Ruan's invariant counted only connected submanifolds.  In a series of fundamental papers (\cite{T4}, \cite{T}), Taubes defined Gromov-Taubes invariants, which give a delicate count of embedded, possibly disconnected submanifolds (This count refines Ruan's earlier results, in particular providing a detailed analysis of the behavior of square 0- tori and their multiple covers.), and equated them with Seiberg-Witten invariants.

In this paper, we define corresponding relative invariants, called relative Ruan invariant and relative Gromov-Taubes invariant, which count embedded symplectic submanifolds which contact a fixed stable symplectic hypersurface $V$ at prescribed points with prescribed contact orders (in addition to insertions on $X\backslash V$).  These invariants will be shown to be deformation invariants of the symplectic structures, note however that due to the relative setting, the symplectic structures must all make the hypersurface $V$ symplectic.  This connects these invariants to the relative symplectic cone defined \cite{DL}.

It should be emphasized, that Ruan-type invariants give a count of geometric objects, much as counting curves in classic algebraic geometry.  GW-type invariants count maps, not the actual geometric images.  In the absolute setting, these two viewpoints were shown to be related to each other by Ionel and Parker, see \cite{IP1}.

%
%

The following Sections are devoted to a precise formulation of the relative spaces of submanifolds underlying the relative Ruan invariant and the properties of these spaces.  Let $X$ be a symplectic 4-manifold and $V$ a 2-dimensional symplectic submanifold.  We would like to count embedded symplectic submanifolds which intersect $V$.  Section \ref{spaces} defines the space of relative, connected submanifolds ${\mathcal K}_V(A,J,\mathcal I_A)$ in class $A\in H_2(X)$, where $A$ is not multiply toroidal, i.e. not of the form $A=mT$ with $T^2=0$ and $K_\omega\cdot T=0$, which meet the hypersurface $V$ according to initial data $\mathcal I_A$.  In order to understand the role of the almost complex structure $J$, we consider the symplectic submanifolds as embedded $J$-holomorphic curves.  The behavior of curves in class $A$ for generic $J$ is rather straightforward if $A\ne \mathfrak V$, denoting the class of $V$ by $\mathfrak V$.  However, care must be taken when $A=\mathfrak V$:  The associated differential operator is no longer surjective if $d_\mathfrak V=\frac{\mathfrak V^2+c_1(\mathfrak V)}{2} <0$.  In this case the curve $V$ is rigid, this leads to the definition of a stable symplectic hypersurface $V$ as one which has $d_{\mathfrak V}\ge 0$.  Generically, for a class $A$, we can avoid this issue; if $A\ne\mathfrak V$, then we simply have no curves in this class generically if  $d_A<0$.  However, this does not apply to the class $\mathfrak V$ as we are free to choose and fix a hypersurface $V$ at the start, so our choice may be non-generic in the sense that $d_\mathfrak V<0$.

The following issue must also be addressed:  Submanifolds which have components that lie in the hypersurface $V$.  This is of interest in particular with a view towards a degeneration formula.  This is discussed preceding Lemma \ref{index}.

Having understood the behavior of embedded curves under deformations of the almost complex structure (Section \ref{generic}), we proceed to understand the properties of the space of relative {\it connected} submanifolds $\mathcal K_V(J,\mathcal I_A)$ for fixed initial data $I_A$ and stable $V$.  We show in Section \ref{structure} that $\mathcal K$ is a smooth finite manifold which behaves well under deformations of the almost complex structure, the initial data and the symplectic structure on $X$.  The greater part of Section \ref{structure} addresses the compactness of $\mathcal K_V$.    

After finding a suitable set of almost complex structures and initial data such that the spaces ${\mathcal K}_V(A,J,\mathcal I_A)$ have the desirable properties, we proceed to define a number $Ru^V(A,[\mathcal I_A])$.  After making precise this definition, we show the invariant properties of $Ru^V(A,[\mathcal I_A])$:  It does not depend on the particular choice of $J$ or $\mathcal I_A$, but rather the class $[\mathcal I_A]$ and the deformation class of $(X,V,\omega)$, see Theorem \ref{invN}.  

At this stage, we will have defined a basic invariant which counts connected non-multiply toroidal curves for stable $V$.  From here, one can proceed in different directions.  A natural inclination would be to attempt to remove the conditions placed on $A$ and $V$.  More general choices of $V$ will involve different methods and will be described in a further paper.  Here we only make some remarks concerning the extension of our results to the non-stable case, see Lemma \ref{wsup}.  A relative Taubes invariant, giving a count of multiply toroidal $A$ relative to $V$,  is discussed in Section \ref{taubesinv}.  Our results on generic almost complex structures in the relative setting and arguments in \cite{T}, show that the invariant defined by Taubes needs no modification in the relative setting.

We proceed by extending the basic invariant to disconnected curves and thereby developing a relative version of the Gromov-Taubes invariant is the content of Section \ref{rgt}.  We show that this invariant can be written as a product of the connected invariants, hence the invariant properties follow immediately.

A further direction would be to incorporate rim-tori.  In \cite{IP4}, Ionel and Parker define a relative Gromov-Witten invariant which takes into account so called rim-tori structures.  Including such structures into our count, we construct a more refined invariant than $Ru^V(A,[\mathcal I_A])$.  In Section \ref{rimtori} we review their rim-tori constructions and define $Ru^V(\hat A,[\mathcal I_A])$.  It is expected that the refined invariant defined in Section \ref{refined} can distinguish isotopy classes of hypersurfaces in the class $\mathfrak V$.

\section{Relative Submanifolds and their Behavior}

In this section we will describe the framework needed to define a space of submanifolds of $X$ which meet a fixed symplectic hypersurface $V$.  We will precisely describe the constraining data on the submanifolds as well as how to understand this from the viewpoint of sections of the bundle $N$ as was used in the absolute case by Taubes.    In contrast to Gromov-Witten theory, we shall fix certain initial data and not view our invariant as an abstract map on cohomology, though we shall allow the initial data to move in homological families.  This initial data shall determine how the submanifolds meet $V$, they can be viewed as representatives of insertions in Gromov-Witten theory.

Let $X$ be a compact, connected 4-dimensional smooth manifold admitting symplectic structures and fix a connected symplectic hypersurface $V\subset X$.  The symplectic form $\omega$ is a nondegenerate closed 2-form on $X$, its class will be denoted by $[\omega]\in H^2(X,\mathbb R)$.    Denote the space of $\omega$-compatible almost complex structures $J$ by $\mathcal J_\omega$.  Given any $J\in\mathcal J_\omega$, denote the corresponding canonical class $-K_\omega = c_1(X,\omega) = c_1(X,J)$.  Note the dependence on the choice of symplectic form.

For any homology class $A\in H_2(X,\mathbb Z)$, define
\begin{equation}
 \label{d}d_A = \frac{A^2-K_\omega\cdot A }{2}.
\end{equation}
Denoting the homology class of $V$ by $\mathfrak V\in H_2(X)$, we define
\begin{equation}
\label{l}l_A=A\cdot \mathfrak V.
\end{equation}

Underlying the constructions in the following sections is the symplectic hypersurface $V$.  The choice of symplectic structure $\omega$ and almost complex structure $J$ must preserve the symplecticity of $V$.  The symplectic structure must be chosen such that $\omega$ is an orientation compatible symplectic structure on $X$ and restricts to an orientation compatible symplectic structure on $V$.  Denote the set of such forms by $\mathcal S_X^V$.  The set of classes $[\omega]\in H^2(X)$ with this property is called the relative symplectic cone $\mathcal C_X^V$ (see \cite{DL}).  

The almost complex structure $J$ must be chosen such that $V$ is a $J$-holomorphic submanifold.
 
\begin{definition}Let $\omega\in \mathcal S_X^V$.  The set $\mathcal J_V\subset \mathcal J_\omega$ is defined to be the set of almost complex structures $J$ making $V$ pseudoholomorphic. 
\end{definition}

This set of almost complex structures will form the basis for our calculations throughout this paper.  That it is rich enough to allow for deformations of pseudoholomorphic curves will be shown in Section \ref{generic}.  We will also encounter the set $\mathcal J_V\times [\mathcal I_A]$; this product we endow with the product topology.  

\subsection{\label{contactorder}Contact Order}

For any given $J\in\mathcal J_V$, the contact order between a $J$-holomorphic submanifold $C$ and the fixed hypersurface $V$ is given as follows (see \cite{M4} and Lemma 3.4, \cite{IP4} for details):  Both submanifolds can be viewed as $J$-holomorphic curves for the given almost complex structure $J\in \mathcal J_V$.  Let $ f:\Sigma\rightarrow C$ be a simple $J$-holomorphic map from a genus $g=g([C])$ Riemann surface to $X$ having as its image the curve $C$.  The genus $g([C])$ is given by the adjunction formula for the curve $C$.  Consider a point of intersection $p\in V$ of $C$ and $V$.    Fix local coordinates $\{vî\}$ in $V$ and let $x$ be local coordinate in normal direction.  Then either $f(\Sigma)\subset V$ or there is an integer $s>0$ and $a_0\in \mathbb C$ such that
\begin{equation}\label{contact}
f(z,\overline z)=(p+ O(\vert z\vert), a_0z^s+O(\vert z\vert ^{s+1})).
\end{equation}     

\begin{lemma}
Assume that $C\not\subset V$ and let $f:\Sigma\rightarrow X$ and $f':\Sigma'\rightarrow X$ be two simple $J$-holomorphic embeddings of $C$.  Then $s=s'$.  
\end{lemma}

\begin{proof}  The maps $f$ and $f'$ have the same image, hence by Cor. 2.5.3, \cite{McS}, there exists a holomorphic map $\phi:\Sigma\rightarrow \Sigma'$ such that $f=f'\circ \phi$.  Inserting this relation into Eq. \ref{contact}, we obtain the fact that $\phi(z)$ must vanish to order 1 at $z=0$.  A series expansion of the holomorphic map $\phi$ in local coordinates as given above, shows that it has leading term $cz$, $c\in\mathbb C$.  Moreover, comparing the second terms from Eq. \ref{contact}, provides $a_0'\phi(z)^{s'} + O(\vert \phi(z)\vert ^{s'+1}))=a_0z^s+O(\vert z\vert ^{s+1}))$ which implies that the leading term in $\phi^{s'}$ must match the leading term $a_0z^s$.  This implies $s=s'$.

\end{proof}
 
\begin{definition}\label{contord}
The contact order of $C$ and $V$ at $p\in V$ for $C\not\subset V$ is defined to be $s$.
\end{definition}

This definition is very intuitive and fits nicely the standard picture of contact order.  However, as seen in the previous Section, to define an invariant, we need to consider an evaluation mapping on sections of the normal bundle to a fixed embedded curve.  In particular, we need to define contact order for sections of the normal bundle $N$.  For that reason, we will adopt the picture presented in \cite{LR}.  We describe the corresponding construction briefly and then show how to obtain an "evaluation " mapping from this picture.

Consider the normal bundle $N$ of $C$.  Denote the boundary of the bundle by $\partial N$.  With a view towards the embedding of the disk bundle $U$ constructed by Taubes, we can associate to the boundary a $S^1$ action, such that $\partial N/S^1=V$.  Let $x(t)$ denote an orbit of the $S^1$ action and call $x(kt)$ a $k$-periodic orbit for any integer $k$.  Remove $V$ from $X$, the resulting "punctured" manifold can be viewed locally in the neighborhood of the removed hypersurface $V$ as $\mathbb R\times \partial N$.  Any $J$-holomorphic curve $u:\Sigma\rightarrow X$ which contacts $V$ can be viewed locally in this picture, and in local coordinates can be written as $u:\Sigma_p\rightarrow \mathbb R \times \partial N$ with $u=(a,u_V)$ from a punctured Riemann surface $\Sigma_p$.  Li-Ruan showed, that as we approach the contact point with $V$,  $u_V\rightarrow x(kt)$ for some $k$-periodic orbit, whereby $k$ is the contact order of $u$ as described previously.  Let $S_k$ denote the space of $k$-periodic orbits, note that we can identify $S_k$ with $V$.

Let $N_V$ be the normal bundle of the hypersurface $V$ and consider a single intersection point $z$ of the curve $C$ and $V$.  In a neighborhood of the point $z$ trivialize the bundle $N_V$.  On the trivialization $U\times F^V_z$ introduce the coordinates $v$ and $x$ as before Eq. \ref{contact}.  Consider the fiber $F^C_z$ of the bundle $N$ over the point $z\in C$.  A section $\mathfrak s\in \Gamma(N)$ intersects $F_z^C$ at a point which we can also identify as belonging to $F_z^V$ over a point $z'\in U$ in the trivialization.  Thus, in local coordinates $(v,x)$ at the intersection point, we can assign a contact order to the section $\mathfrak s$ as given by Def \ref{contord}.  Moreover, by removing the intersection point $\mathfrak s\cap F_z^C$, we can apply the orbit construction and assign to the section a $k$-periodic orbit with the same contact order as the curve $C$.  Thus we define a map 
\begin{equation}\label{contactmap}
G_k:\Gamma(N)\rightarrow S_k
\end{equation}
 for fixed $k$ determined by $C$.  This will be used to define an evaluation map needed to define the relative invariant in Section \ref{relinv}.

\subsection{Initial Data}

A submanifold $C\subset X$ will be constrained by two types of data:  First, we fix a set of geometric objects, i.e. points, curves, etc., which the submanifold must contact.  Secondly, at each of the contact points with the fixed hypersurface $V$, we prescribe as well the contact order $s\in\mathbb N$ of the two submanifolds.  This contact order will be defined more precisely in \ref{contactorder}.  We collect them in the initial data $\mathcal I_A$:

\begin{definition}  The initial data ${\mathcal I}$ is defined as follows:  Fix integers $d,l\in \mathbb Z$.  For $d , l\ge 0$, choose decompositions $d_1+d_2=d$ and $l_1+l_2+l_3=l$ with $d_*,l_*\ge 0$.  Then $\mathcal I$ consists of the following sets:

\begin{enumerate}
\item  $\Omega_{d_1}\subset X\backslash V$ consisting of $d_1$ distinct points,
\item $\Omega_{l_1}$ a collection of $l_1$ pairs $(x,s)\in V\times {\mathbb N}$ with all $x$ distinct,
\item  $\Gamma_{d_2}$ a collection of $d_2$ disjoint 1-dimensional submanifolds of $X\backslash V$, 
\item $\Gamma_ {l_2}$ a collection of $l_2$ pairs $(\gamma,s)$ with $\gamma$ a  1-dimensional submanifold of $V$ and $s\in {\mathbb N}$, all $\gamma$ pairwise disjoint,
\item $\Upsilon=\Upsilon_V^{l_3}$ a collection of $l_3$ pairs $(V,s)$ of copies of the hypersurface $V$ and $s\in\mathbb N$.  \end{enumerate}

If $d_*=0$ or $l_*=0$, then the respective sets are the empty set.  If $d< 0$, then $\mathcal I=\emptyset$.  If $l< 0$, then $\Omega_{l_1}=\Gamma_{l_2}=\Upsilon=\emptyset$.

For a class $A\in H_2(X)$, we denote by $\mathcal I_A$ any set of initial data $\mathcal I$ with $d=d_A$ and $l=l_A$.

\end{definition}

The collection $\Upsilon$ represents  insertions which offer no constraint on the intersection point of the submanifold and $V$ other than a certain contact order at the intersection point.

\begin{definition}
Initial data $\mathcal I_A$ will be called proper initial data for the class $A$ if the following hold:
\begin{itemize}
\item If $d_A\ge 0$, then
\begin{enumerate}
\item $0\le d_1\le d_A$, $0\le d_2\le 2d_A$ and $2d_1+d_2\le 2d_A$,
 \item $2d_1+d_2-l_2-2l_3=2(d_A-l_A) $ and
\item $A\cdot \mathfrak V=\sum s_i$
\end{enumerate}
\item If $d_A < 0$, then $\mathcal I_A=\emptyset$. 
\end{itemize}
\end{definition}

This definition is motivated by viewing these sets as the analogue to cohomological insertions in classical Gromov-Witten theory.  A dimension count shows that we cannot expect to have any submanifolds in class $A$ satisfying the initial data $\mathcal I_A$ unless $2d_1+d_2-l_2-2l_3=2(d_A-l_A)$.  Note further, that the condition imposed on the orders $s_i$ constrains the value of $l_1$:  $l_1+l_2+l_3\le A\cdot \mathfrak V=l_A$.  Moreover, this condition also ensures that if $A\cdot \mathfrak V>0$, then some $l_*$ will be nonzero.  The following observation will simplify the multiply toroidal case: 

\begin{lemma}\label{values}Let  $d_A=0$. The initial data $\mathcal I_A$ is proper if and only if
\begin{enumerate}
\item $d_1=d_2=l_1=l_2=0$
\item $l_3=A\cdot \mathfrak V$ and all values of $s_i=1$.
\end{enumerate}
\end{lemma}

\begin{proof}If $d_1=d_2=l_1=l_2=0$ and $l_3=A\cdot \mathfrak V=l_A$, then the first two conditions for proper hold.  Moreover, $s_i=1$ and $l_3=l_A$ imply $A\cdot \mathfrak V=\sum s_i$.

The definition of proper shows that $d_A=0$ is equivalent to $d_1=d_2=0$.  Furthermore, we obtain that 
\[
-l_2-2l_3=-2l_A\le -2l_1-2l_2-2l_3\;\Rightarrow \;-l_2\ge 2l_1\;\Rightarrow\;l_1=l_2=0
\]
and thus $l_3=l_A$.  Then $s_i=1$ follows from $A\cdot \mathfrak V=\sum_{i=1}^{l_3} s_i$.

\end{proof}

This result can be reformulated as follows: If $d_A=0$, then for proper initial data we have $s_i=1$ for all contacts, the only contacts being allowed are such, that no constraints are imposed on a curve in class $A$.  However, a curve with $A\cdot \mathfrak V\ne 0$ (topological intersection number!) will generically intersect the hypersurface $V$ in exactly $l_A$ points each with order 1.  Hence this result simply states that a curve which has no restriction $d_A$ will not allow any restrictions in its contact with $V$.

Our constructions will involve operators on a universal space parametrized by almost complex structures and initial data.  We will only be interested in initial data which varies in a specified manner.  This is made precise in the following definition:

\begin{definition}  Fix a set of initial data $\mathcal I_A$.  Define the class $[\mathcal I_A]$ as the set of initial data with the following properties:
\begin{enumerate}
\item $\mathcal I_A\in[\mathcal I_A]$
\item All initial data have the same values of $d_*,l_*$ and $s_*$.
\item Corresponding elements in the sets $\Omega_*$, $\Gamma_*$ and $\Upsilon$ have the same values of $s_i$ and lie in the same homology class.
\item The ordering of the elements in the $\Gamma$-sets of any initial data is the same as in the given set $\mathcal I_A$.  
\end{enumerate}

If $\mathcal I_A$ is proper, then the class $[\mathcal I_A]$ is called proper.
\end{definition}

Note that it is not necessary to give an initial set $\mathcal I_A$.  Simply fixing the values of $d_*,l_*$ and $s_*$ as well as the relevant homology classes and an ordering will define a class $[\mathcal I_A]$.  We will consider data from this viewpoint.

\subsection{Non-Degenerate Submanifolds}

Consider an embedded submanifold $C\subset X$ together with an integer $m$ and fix $J\in \mathcal J_V$.  We can view the normal bundle $N$ of $C$ in $X$ as a complex bundle with complex structure induced by the almost complex structure on $X$.  On the other hand, we can construct a disk bundle $U$ over $C$ with complex structure $J_0$ induced from the restriction of the almost complex structure on $X$ to $C$ by setting $J_0\vert_{\mbox{fiber F}}=J_{\pi(F)}$ with $\pi:U\rightarrow C$.  Taubes has constructed an embedding of this disk bundle $U$ into the normal bundle $N$ which is uniquely associated to the submanifold $C$ and which allows for a comparison of the two complex structures.  This leads to two complex valued sections $\nu \in T^{0,1}C$ and $\mu \in T^{0,1}C\otimes N^{\otimes 2}$ associated to $C$.  These define an operator
\begin{equation}
\label{D}Ds=\overline\partial s + \nu s+\mu \overline s
\end{equation}
which is a compact perturbation of the standard $\bar\partial$ operator (as defined by the complex structures on the domain and target) and hence is elliptic.   Moreover, this operator is canonically associated to the submanifold $C$.  Its kernel can be viewed as the tangent space to the pseudo-holomorphic embeddings of $C$ into $X$ in the space of all smooth embeddings of $C$ in $X$.  This leads to  the definition of non-degenerate:

\begin{definition}\label{nondeg}Fix an almost complex structure $J\in \mathcal J_V$ and a class $A\in H_2(X)$ with $d_A\ge 0$.  Choose initial data $\mathcal I_A$.  Let  $C$ be a connected, $J$-holomorphic submanifold in $X$ such that $\mathcal I_A\subset C$.  $C$ is non-degenerate if the following hold: 
\begin{enumerate}
\item If $d_A=0$, then cokernel$(D)=\{0\}$.
\item If $d_A>0$, then $D\oplus ev_{\mathcal I_A}$ has trivial cokernel.  Here $ev_{\mathcal I_A}$ is the evaluation map which takes a smooth section of $N$ to its value over the data in $\mathcal I_A$. 
\end{enumerate}

\end{definition}

\subsubsection{Non-Degeneracy for classes with $d_A<0$} 

The definition of non - degeneracy as given in Def \ref{nondeg} allows only for classes with $d_A\ge 0$.  In the relative setting, we must allow for a further case:  Assume the hypersurface $V$ is non-stable, i.e. the class $\mathfrak V$ has $d_\mathfrak V<0$.  Then we could still choose $A=\mathfrak V$ and consider curves in this class relative to $V$.  Definition \ref{nondeg} can be extended to include this case as follows:non-stable, i.e. the class $\mathfrak V$ has $d_\mathfrak V<0$.  Then we could still choose $A=\mathfrak V$ and consider curves in this class relative to $V$.  Definition \ref{nondeg} can be extended to include this case as follows:

\begin{definition}\label{nondegng}Fix an almost complex structure $J\in\mathcal J_V$.  If $d_A<0$, then a submanifold $C$ representing the class $A\in H_2(X)$ is called non-degenerate if it is rigid and there exist no other $J$-holomorphic curves $C'$ in the class $A$.

\end{definition}

We will show, that this only applies in the setting mentioned above.  If $A\ne \mathfrak V$, then we will not be able to find a curve $C$ in the class $A$ for generic almost complex structures $J$.  Note also, that in the relative setting, we allow only almost complex structures in $\mathcal J_V$ for the definition of nondegeneracy.

\subsection{\label{spaces}The Space of Relative Submanifolds}

We now introduce the space of relative submanifolds ${\mathcal K}_V(A,J,{\mathcal I_A})$ for non multiply toroidal classes $A$, i.e. classes which are not of the form $A=mT$ for some $T$ with $T^2=0$ and $K_\omega\cdot T=0$.  This definition will be rather technical, however the general idea is simple:  We want to consider all connected submanifolds $C$, which contact $V$ in a very controlled manner.  This is determined by the initial data $\mathcal I_A$ and we ensure that we contact $V$ only once for every given geometric object with the required contact order.  Moreover, the curve $C$ shall meet each geometric object in the initial data $\mathcal I_A$.  We make this precise in the following definition:

\begin{definition}\label{defK}Fix $A\in H_2(X)$ and a set of proper initial data $\mathcal I_A$.  Assume that $A$ is not multiply toroidal.  Choose an almost complex structure $J\in\mathcal J_V$.  Denote the set $\mathcal K_V= {\mathcal K}_V(A,J,\mathcal I_A)$ of connected $J$-holomorphic submanifolds $C\subset X$ which satisfy

\begin{itemize}
 
\item If $d_{A}>0$, then $C$  
\begin{enumerate}
\item  contains $\Omega_{d_1}$ and
\item  intersects each member of $\Gamma^X$ exactly once.
\end{enumerate}
\item If $l_{A}=0$, then $C\cap V=\emptyset$.
\item If $l_{A}>0$, then $C$
\begin{enumerate}
\item intersects $V$ locally positively and transversely,
\item intersects $V$ at precisely the $l_1$ points of $\Omega_{l}$ and
\item  intersects each member of $\Gamma^V$ exactly once.  
\item The remaining $l_3$ intersections with $V$ are unconstrained.
\item Each intersection is of order $s_i$ given in the initial data $\mathcal I_A$ for this component. 
\end{enumerate}
\end{itemize}
\end{definition}

\subsection{\label{mapinto}Convergence of Relative Submanifolds}  

We are interested primarily in curves which genuinely intersect $V$, transversely and locally positively.  Moreover, we will eventually prove a sum formula similar to results in \cite{LR} or \cite{IP5}.  This will be simplified considerably by the exclusion of components in $V$.

{\bf Convergence Behavior.} In order to define a relative invariant, we will need to understand the compactness properties of the relative spaces $\mathcal K$.  This will be done in Chapter \ref{structure}.  In order to simplify calculations in these sections, we consider here the behavior of curves descending into $V$ under convergence and determine their index.  

Any sequence of symplectic submanifolds will converge to a limit curve by Gromov compactness.  However, this limit curve may have components mapping into $V$.  We now describe how to handle such curves, this is described in detail in \cite{McD1} and \cite{HLR} (see also \cite{LR} and \cite{Li1}, \cite{McD1} contains numerous examples of this construction).  

The idea is to extend the manifold $X$ in such a manner, that the components descending into $V$ get stretched out and become discernible.  This extension is achieved by gluing $X$ along $V$ to the projective completion of the normal bundle $N_V$.  This completion is denoted by $Q=\mathbb P(N_V\oplus \mathbb C)$ and it comes with a natural fiberwise $\mathbb C^*$ action.  The ruled surface $Q$ contains two sections, the zero section $V_0$, which has opposite orientation to $V$, and the infinity section $V_\infty$, which is a copy of $V$ with the same orientation both of which are preserved by the $\mathbb C^*$ action.  The manifold $X\#_{V=V_0}Q$ is symplectomorphic to $X$ and can be viewed as a stretching of the neighborhood of $V$.  This stretching can be done any finite number of times.  Therefore consider the singular manifold $X_m=X\sqcup_{V=V_0}Q_1\sqcup_{V_\infty=V_0}...\sqcup_{V_\infty=V_0}Q_m$ which as been stretched $m$ times.  This will provide the target for the preglued submanifolds which we now describe.  

Any curve in $X$ with components lying in $V$ can be viewed as a submanifold in $X_m$ consisting of a number components:  Each such curve has levels $C_i$ which lie in $Q_i$ (denoting $X=Q_0$) and which must satisfy a number of contact conditions.  $C_i$ and $C_{i+1}$ contact along $V$ in their respective components $(Q_i,V_\infty)$ and $(Q_{i+1},V_0)$ such that contact orders and contact points match up.  The imposed contact conditions on $V$ from the initial data $\mathcal I_A$ are imposed on the level $C_m$ where it contacts $V_\infty$ of $Q_m$ whereas the absolute data is imposed on $C_0$.  

The submanifolds $C_i$, viewed as maps into $X_m$, must satisfy certain stability conditions.  In $X$, these are the well known standard conditions on the finiteness of the automorphism group.  For those mapping into $Q_i$, $i>0$, we identify any two submanifolds which can be mapped onto each other by the $\mathbb C^*$ action of $Q$.

After constructing $\{C_i\}$ in $X_m$ we obtain a genuine curve in $X$ meeting $V$ as prescribed by $\mathcal I_A$ by gluing along $V$ in each level.  If each level was embedded, then so will the glued curve be.  To show that each level remains embedded, i.e. we obtain no nodes away from the sections in $Q$, will be part of the task of the later sections. 

The homology class of the preglued curve is defined as the sum of the homology class of $C_0$ and the projections of the class of $C_i$ into $H_2(V)$.

For each such curve $C$, it is possible to determine the index of the associated differential operator.  This is of course of central importance when determining the dimension of the spaces $\mathcal R$.  The following Lemma sums up the result.

\begin{lemma}\label{index}(Lemma 7.6, \cite{IP4}; \cite{LR})  Let $C$ be a preglued submanifold with $m+1$ levels representing the class $A\in H_2(X)$ which meets the data in $\mathcal I_A$ in the prolongation $X_m$ as described above.  Then the index of $C$ is
\begin{equation}
ind(C)= d_A-m.
\end{equation}
\end{lemma}

This result shows, that for a generic choice of almost complex structure and initial data, we do not expect to have curves admitting higher levels. 

\begin{remark}
In the calculation of the invariants defined in the later sections of this paper, as well as in Gromov-Witten theory, it is not always convenient to use a generic choice of almost complex structure or initial data.  For this reason, it would be of use to be able to exclude higher level curves for all pairs $(J,\mathcal I_A)$.  A simple method for ensuring no component maps non-trivially into $V$ is to consider classes such that the genus $g(A)$ given by the adjunction equality satisfies 
\[
g(V)>g(A).
\]

Further, one could try to use the following result in place of the condition on $A$:

\begin{lemma}\label{wsup}
Let $(X,\omega)$ be a symplectic 4-manifold, $V$ a symplectic hypersurface.  Let $A\in H_2(X)$ and assume that for some pair $(J,\mathcal I_A)$ the set $\mathcal R_V(A,J,\mathcal I_A)\ne\emptyset$.  Assume further that $A\cdot\mathfrak V> A^2\ge 0$.  Then there exists a symplectic form $\tilde \omega\in \mathcal S_X^V$ such that $[\tilde\omega]\cdot \mathfrak V> [\tilde\omega]\cdot A$.  Moreover, $\tilde\omega$ is a smooth deformation of $\omega$ through symplectic forms.
\end{lemma}

\begin{proof}
This is an application of Lemma 2.1 in \cite{Bi}.  Our assumptions imply the existence of a curve $C$ in class $A$ which intersects $V$ transversally and locally positively in a finite number of points.  Then \cite{Bi} has shown, that there exists a symplectic form in the class $[\omega(t,s)]=t[\omega]+sA$, $(t>0, s\ge 0)$, which makes $V$ symplectic.  In particular, for $t$ small and $s$ large enough we obtain 
\[
[\omega(t,s)]\cdot \mathfrak V > [\omega(t,s)]\cdot A.
\]

\end{proof}

This Lemma implies, that if $A\cdot\mathfrak V> A^2\ge 0$ holds, then we can find a relative symplectic form on $(X,V)$ such that the symplectic area of $V$ is larger than the area of any curve in class $A$.  This then precludes any components of a curve in class $A$ lying in $V$.

\end{remark}

\subsection{Main Result}

The next sections are devoted to the proof of the following Proposition:

\begin{prop} \label{mainprop}Fix a class $A\in H_2(X)$ and a proper class $[\mathcal I_A]$.  Assume $V$ is a symplectic hypersurface and $A$ is not multiply toroidal.  Then there exists a Baire subset of $\mathcal J_V\times [\mathcal I_A]$ such that

\begin{enumerate}
\item The set ${\mathcal K}_V(A,J,{\mathcal I_A})$ is a finite set.
\item If $A\ne \mathfrak V$, then ${\mathcal K}_V(A,J,{\mathcal I_A})$ is empty when $d_A<0$. 
\item If $V$ is an exceptional sphere, then ${\mathcal K}_V(\mathfrak V,J,\emptyset)=\emptyset$
\item Every point $h\in {\mathcal K}_V$ has the property, that each $C$ with is non-degenerate unless possibly if $C$ is a torus with trivial normal bundle.  In this case it is $m$-non-degenerate for all $m>0$.
\item If $(J^1,\mathcal I_A^1)$ are sufficiently close to $(J,\mathcal I_A)$, then the sets ${\mathcal K}_V$ and ${\mathcal K}_V^1$ have the same number of elements.
\end{enumerate}
\end{prop}

{\bf Remark:}   Prop. \ref{mainprop} does not follow immediately from Taubes' results, as the relationship between the sets of generic almost complex structures in ${\mathcal J}_\omega$ and ${\mathcal J}_V$ given in Prop. 7.1, \cite{T1} and Prop. \ref{mainprop} is unclear.

\section{Generic Almost Complex Structures}

In this section we will show that the set $\mathcal J_V$ is rich enough to allow for deformations of embedded symplectic submanifolds.  To do so, we will define a suitable universal space $\mathcal U$ and the set of connected submanifolds $\mathcal K$.  We show that the set $\mathcal K$ can be described as the zero set of a suitable section $\mathcal F$ of a bundle $\mathcal B$ over $\mathcal U$ and that $\mathcal F$ behaves as expected at its zeros.

\subsection{The Universal Model}

Fix $A\in H_2(X)$ and a symplectic form $\omega\in \mathcal S_X^V$.  Let $\Sigma$ be a compact, connected, oriented 2-dimensional surface of genus $g=g(A)$ as defined by the adjunction formula.    Let $J\in \mathcal J_V$ and consider $d_A$ and $l_A$ as defined in \ref{d} and \ref{l} resp.  Let $s\in {\mathbb N}^l$ be a vector of length $l\le l_A$ such that $\sum s_i=A\cdot \mathfrak V$. (Clearly, if $l>l_A$, then we cannot expect under generic conditions to have any submanifolds in class $A$.)  When $d_A$ and $l_A$ are nonnegative, choose initial data ${\mathcal I_A}=(\Omega,\Gamma,\Upsilon)$ using the values in $s$.    Introduce $\mathcal K_V(J,{\mathcal I_A})=\mathcal K$ as the set of $J$-holomorphic submanifolds in $X$ which 
\begin{itemize}
\item are abstractly diffeomorphic to $\Sigma$, 
\item meet the data in $\mathcal I_A$ as described in Def \ref{defR} and 
\item have fundamental class $A$.

\end{itemize}

The space $\mathcal K$ is essentially $\mathcal K_V$, just that we have changed the viewpoint from abstract submanifolds to those diffeomorphic to a fixed $\Sigma$.  For these reasons, a good understanding of the properties of $\mathcal K$ is necessary to prove Prop. \ref{mainprop}.

Let $\mathcal J_V\times [\mathcal I_A]$ be the parameter space for the universal model to be defined below and corresponding to the class of initial data in $\mathcal K_V(J,{\mathcal I_A})$.

A universal space $\mathcal U$ for $\mathcal K_V(J,\mathcal I_A)$ consists of Diff$(\Sigma)$ orbits of a 4-tuple $(i,u,J,\mathcal I_A)$ with 
\begin{enumerate}
\item $u:\Sigma\rightarrow X$ an embedding off a finite set $\mathcal E$ of points from a Riemann surface $\Sigma$ such that $u_*[\Sigma]=A$ and $u\in W^{k,p}(\Sigma,X)$ with $kp>2$,
\item $i$ a complex structure on $\Sigma$ and $J\in \mathcal J_V$,
\item $\mathcal I_A\in[\mathcal I_A]$ and $\mathcal I_A\subset u(\Sigma)$.
\end{enumerate}
Note that every map $u$ is locally injective.

The last condition needs some explaining:  The initial data $\mathcal I_A$ consists of two types of sets:  Sets contained in $X\backslash V$ and pairs consisting of points in $V$ and an integer $s$.  The first set should be contained in $u(\Sigma)$, meaning the image goes through the constraints on $X\backslash V$, meeting each curve in $\Gamma_{d_2}$ only once.  The second set should also be contained in the image $u(\Sigma)$, however, each point should have prescribed contact order $s$.  Furthermore, the image should meet each element in $\Gamma_{l_2}$ exactly once.  These are exactly the conditions imposed in the definition of the space $\mathcal K$.

\subsection{\label{generic}Generic Complex Structures in $\mathcal J_V$}

We wish to show that $\mathcal J_V$ has a rich enough structure to allow for genericity statements for $J$-holomorphic curves.   A portion of these results has appeared in an Appendix in \cite{DL}. 

If $A\ne\mathfrak V$, the genericity results are proven by the standard method:  We will define a map $\mathcal F$ from the universal model $\mathcal U$ to a bundle $\mathcal B$ with fiber $W^{k-1,p}(\Lambda^{0,1}T^*\Sigma\otimes u^*TX)$ over $(i,u,J,\mathcal I_A)$ and show that it is submersive at its zeros.  Then we can apply the Sard-Smale theorem to obtain that $\mathcal J_V^A$ is of second category.  This will involve the following technical difficulty:  The spaces $\mathcal J_V$ and any subsets thereof which we will consider are not Banach manifolds in the $C^\infty$-topology.  However, the results we wish to obtain are for smooth almost complex structures.  In order to prove our results, we need to apply Taubes trick (see \cite{T} or \cite{McS}):  This breaks up the set of smooth almost complex structures into a countable intersection of sets, each of which considers only curves satisfying certain constraints.  These subsets are then shown to be open and dense by arguments restricted to $C^l$ smooth structures, where the Sard-Smale theorem is applicable.  We will not go through this technical step but implicitly assume this throughout the section, details can be found in Ch. 3 of \cite{McS}.

\begin{lemma}\label{genericA}
Let $A\in H_2(X,\mathbb Z)$, $A\ne \mathfrak V$, and let $\mathcal I_A$ be a set of initial data. Denote the set of pairs $(J,\mathcal I_A)$ by $\mathfrak I$.  Let $\mathcal J_V^A$ be the subset of pairs $(J,\mathcal I_A)$ which are non-degenerate for the class $A$ in the sense of Def. \ref{nondeg}. Then $\mathcal J_V^A$ is a set of second category in $\mathfrak I$.
\end{lemma}

Note that the universal model excludes multiple covers of the hypersurface $V$ in the case that $A=a\mathfrak V$ for $a\ge 2$, and we can thus assume that any map $u:\Sigma\rightarrow M$ satisfies $u(\Sigma)\not\subset V$ for this proof.  In particular, $V$ could be a square 0-torus and $A=a\mathfrak V$ for $a\ge 2$.

\begin{proof}

Define the map $\mathcal F:\mathcal U\rightarrow \mathcal B$ as $(i,u,J,\mathcal I_A)\mapsto \overline\partial_{i,J}u$.  Then the linearization at a zero $(i,u,J,\mathcal I_A)$ is given as
\begin{equation}
{\mathcal F}_*(\alpha,\xi,Y)=D_u\xi + \frac{1}{2}(Y\circ du\circ i + J\circ du\circ \alpha)
\end{equation}
where $D_u$ is Fredholm, $Y$ and $\alpha$ are variations of the respective almost complex structures.  This is a map on $H_i^{0,1}(T_\mathbb C\Sigma)\times W^{k-1,p}(u^*TX)\times T\mathcal J_V$.  The components of $(\alpha,\xi,Y)$ correspond to perturbations in the complex structure $i$, the image $u(\sigma)$ and $J$.

Consider $u\in \mathcal U$ such that there exists a point $x_0\in \Sigma$ with $u(x_0)\in X\backslash V$ and $du(x_0)\ne 0$ (The second condition is satisfied almost everywhere, as $u$ is a $J$-holomorphic map.).  Then there exists a neighborhood $N$ of $x_0$ in $\Sigma$ such that 
\begin{enumerate}
\item $du(x)\ne 0$, 
\item $u(x)\not\in V$ for all $x\in N$.
\end{enumerate}
In particular, we know that the map $u$ is locally injective on $N$.  Furthermore, we can find a neighborhood in $N$, such that there are no constraints on the almost complex structure $J\in \mathcal J_V$, i.e. this neighborhood does not intersect $V$.  More precisely, $Y$ can be chosen as from the set of $\omega$-tame almost complex structures with no restrictions given by $V$.  Denote this open set by $N$ as well.

Let $\eta\in $coker$ \mathcal F_*$.  Consider any $x\in N$ with $\eta(x)\ne 0$.  Then Lemma 3.2.2, \cite{McS}, provides a matrix $Y_0$ with the properties
\begin{itemize}
\item $Y_0=Y_0^T=J_0Y_0J_0$ with $J_0$ the standard almost complex structure in a local chart and 
\item $Y_0[ du(x)\circ i(x)]=\eta(x)$.
\end{itemize}
On $N$ choose any variation $Y$ of $J$ such that $Y(u(x))=Y_0$.  Then define the map $f:N\rightarrow \mathbb R$ by $\langle Y\circ du\circ i,\eta\rangle$.  Note that $f(x)>0$ by definition of $Y$.  Therefore, we can find an open set $N_1$ in $N$ such that $f>0$ on that open set. 
Using the local injectivity of the map $u$ and arguing as in \cite{McS}, we can find a neighborhood $N_2\subset N_1$ and a neighborhood $U\subset M$ of $u(x_0)$ such that $u^{-1}(U)\subset N_2$.  Choose a cutoff function $\beta$ supported in $U$ such that $\beta(u(x))=1$.  Hence in particular
\begin{equation}
\int_\Sigma \langle \mathcal F_*(0,0,\beta Y),\eta\rangle>0
\end{equation}  
and therefore $\eta(x)=0$.  This result holds for any $x\in N$, therefore $\eta$ vanishes on an open set.

As we have assumed $\eta\in $coker$ \mathcal F_*$, it follows that
\[
0\;=\;\int_\Sigma \langle \mathcal F_*(0,\xi,0),\eta\rangle\;=\;\int_\Sigma\langle D_u\xi,\eta\rangle
\]
for any $\xi$.  Then it follows that $D_u^*\eta =0$ and $0=\triangle \eta + l.o.t.$.  Aronszajn's theorem allows us to conclude that $\eta=0$ and hence $\mathcal F_*$ is surjective.

Thus we have the needed surjectivity for all maps admitting $x_0$ as described above: $u(x_0)\not\in V$ and $du(x_0)\ne 0$.  As stated before, this last condition is fulfilled off a finite set of points on $\Sigma$.  The first holds for any map $u$ in class $A$ as we have assumed that $A\ne \mathfrak V$.

Now apply the Sard-Smale theorem to the projection onto the last two factors of $(i,u,J,\mathcal I_A)$ (If $\mathcal I_A=\emptyset$ then only onto the $J$-factor.).
\end{proof}

As we have seen in the above proof, for the class $A=\mathfrak V$ which may have representatives which do not lie outside of $V$, we must be careful.  In particular, it is conceivable, that the particular hypersurface $V$ chosen may not be generic in the sense of Taubes, i.e. the set $\mathcal J_V$ may contain almost complex structures for which the linearization of $\overline \partial_J$ at the embedding of $V$ is not surjective.  The rest of this section addresses this issue.  We begin by showing that the cokernel of the linearization of the operator $\overline\partial_J$ at a $J$-holomorphic embedding of $V$ has the expected dimension:

Let $j$ be an almost complex structure on $V$.  Define $\mathcal J_V^j=\{J\in\mathcal J_V\;\vert \;J\vert_V=j\}$ and call any $J$-holomorphic embedding of $V$ for $J\in  \mathcal J_V^j$ a $j$-holomorphic embedding.

\begin{lemma}\label{genericV}Fix a $j$-holomorphic embedding $u:(\Sigma,i)\rightarrow (X,J)$ for some $J\in\mathcal J_V^j$.    If $d_{\mathfrak V}\ge 0$, then there exists a set $\mathcal J_V^{g,j}$ of second category  in $\mathcal J_V^j$ such that for any $J\in \mathcal J_V^{g,j}$ the linearization of $\overline \partial_{i,J}$ at the embedding $u$ is surjective.  If $d_{\mathfrak V}<0$, then then there exists a set $\mathcal J_V^{g,j}$ of second category  in $\mathcal J_V^j$ such that the hypersurface $V$ is rigid in $X$. 
\end{lemma}

Let us describe the proof before giving the exact proof.  We follow ideas of Section 4, \cite{U}.  We need to show that for a fixed embedding $u:\Sigma\rightarrow X$ of $V$ the linearization $\mathcal F_*$ of $\overline \partial_{i,J}$ at $u$ has a cokernel of the correct dimension for generic $J\in \mathcal J_V^j$.  To do so, we will consider the operator $\mathcal G(\xi,\alpha,J):= \mathcal F_*(\alpha,\xi,0)$ at $(i,u,J,\mathcal I_A)$.  We will show that the kernel of the linearization $\mathcal F_*$ for non-zero $\xi$ has the expected dimension for generic $J$ and hence the linearization of $\overline \partial_{i,J}$ at $u$ also has the expected dimension.  Note also, that for any $J\in\mathcal J_V^j$, the map $u$ is $J$-holomorphic.

What is really going on in this construction?  The operator $\mathcal F$ is a section of a bundle $\mathcal B$ over $\mathcal U$, as described above.  The linearization $\mathcal F_*$ is a map defined on $H_i^{0,1}(T_\mathbb C\Sigma)\times W^{k-1,p}(u^*TX)\times T\mathcal J_V$.  In our setup, we fix the complex structure along $V$ and do not allow perturbations of this structure on $V$.  Hence we remove the infinite dimensional component of the domain of $\mathcal F_*$ and are left with a finite dimensional setup.

Further, we consider a map $\mathcal U\rightarrow \mathcal J_V^j$.  In this map, we fix a "constant section" $u$ over $\mathcal J^j_V$, i.e. we consider the structure of the tangent spaces along a fixed $j$-holomorphic map $u$  while not allowing the almost complex structure $j$ along $V$ to vary.  Note that it is  $j$ which makes $u$ pseudoholomorphic.  Hence fixing $u$ is akin to considering a constant section in the bundle $\mathcal U\rightarrow \mathcal J_V^j$.  

We are only interested in the component of the tangent space along this section, this corresponds to the tangent space along the moduli space $\mathcal M=\mathcal F^{-1}(0)$ at the point $(u,J)$.  However, this is precisely the component of the kernel of $\mathcal F_*$ with $Y=0$, i.e. the set of pairs $(\xi,\alpha)$ such that $\mathcal F_*(\xi,\alpha,0)=0$, which corresponds to exactly the zeros of $\mathcal G$. 

When considering the zeros of the map $\mathcal G$ viewed over $\mathcal J_V^j$, we know from the considerations above that this is a collection of finite dimensional vector spaces.  We may remove any part of these spaces, so long as we leave an open set, which suffices to determine the dimension of the underlying space.  Hence, removing $\xi=0$, a component along which we cannot use our methods to determine the dimension of the kernel, still leaves a large enough set to be able to determine the dimension of the moduli space $\mathcal M$.  

For this reason, we want to show that the kernel of the linearization $\mathcal F_*$ for non-zero $\xi$,or equivalently the zero set of $\mathcal G$ for non-zero $\xi$, has the expected dimension $\max\{d_{\mathfrak V}, 0\}$ for generic $J\in\mathcal J_V^j$.

\begin{proof}The operator $\mathcal G$ is defined as 
\[
W^{k-1,p}(u^*TX)\times H^{0,1}_i(T_{\mathbb C}\Sigma)\times \mathcal J_V^j\rightarrow L^p(u^*TX\otimes T^{0,1}\Sigma)
\]
\[
(\xi,\alpha,J)\mapsto D_u^J\xi + \frac{1}{2}J\circ du\circ \alpha
\]  
where the term $D_u^J=\frac{1}{2}(\nabla \xi+ J\nabla\xi\circ i)$ for some $J$-hermitian connection $\nabla$ on $X$, say for example the Levi-Civita connection associated to $J$.  


Let $(\xi,\alpha,J)$ be a zero of $\mathcal G$.  Linearize $\mathcal G$ at $(\xi,\alpha,J)$:
\[
\mathcal G_*(\gamma,\mu,Y)=D_u^J\gamma +\frac{1}{2}\nabla_\xi Y\circ du\circ i+\frac{1}{2}J\circ du\circ \mu.
\]
As stated above, we assume nonvanishing $\xi$, hence we can assume that $\xi\ne 0$ on any open subset.  Let $\eta\in $coker $ \mathcal G_*$.  Let $x_0\in \Sigma$ be a point with $\eta(x_0)\ne 0\ne \xi(x_0)$.  In a neighborhood of $u(x_0)\in V$ the tangent bundle $TX$ splits as $TX=N_V\oplus TV$ with $N_V$ the normal bundle to $V$ in $X$.  With respect to this splitting, the map $Y$ has the form
\[
Y=\left(\begin{array}{cc} a&b\\0&0\end{array}\right)
\]
with all entries $J$-antilinear and $b\vert_V=0$, thus ensuring that $V$ is pseudoholomorphic and accounting for the fact that we have fixed the almost complex structure along $V$.  Thus $\nabla_\xi Y$ can have a similar form, but with no restrictions on the vanishing of components along $V$.  

Assume $\eta$ projected to $N_V$ is non-vanishing  Then we can choose 
\[
\nabla_\xi Y=\left(\begin{array}{cc} 0&B\\0&0\end{array}\right)
\]
at $x_0$ such that $B(x_0)[du(x_0)\circ i(x_0)](v)=\eta^{N_V}(x_0)(v)$ and $B(x_0)[du(x_0)\circ i(x_0)](\overline v)=\eta^{N_V}(x_0)(\overline v)$ for a generator $v\in T^{1,0}_{x_0}\Sigma$ and where $\eta^{N_V}$ is the projection of $\eta$ to  $N_V$.  Then, using the same universal model as in the previous Lemma, we can choose neighborhoods of $x_0$ and a cutoff function $\beta$ such that 
\[
\int_\Sigma \langle \mathcal G_*(0,0,\beta Y),\eta\rangle >0
\]
and thus any element of the cokernel of $\mathcal G_*$ must have $\eta^{N_V}=0$.  An argument in \cite{U} shows that the projection of $\eta$ to $TV$ must also vanish.  Therefore the map $\mathcal G_*$ is surjective at the embedding $u:\Sigma\rightarrow V$.  

Thus the set $\{(\xi,\alpha,J)\vert \mathcal G(\xi,\alpha,J)=0,\;J\in\mathcal J_V^j,\;\xi\ne 0\}$ is a smooth manifold and we may project onto the last factor.  Then applying Sard-Smale, we obtain a set $\mathcal J_V^{g,j}$ of second category in $\mathcal J_V^j$, such that for any $J\in \mathcal J_V^{g,j}$, the kernel of the linearization of $\overline\partial $ at non-zero perturbations $\xi$ of the map $u$ is a smooth manifold of the expected dimension.  In the case $d_{\mathfrak V}\ge 0$, this however implies that $\mathcal F_*$ at $(i,u,J,\Omega)$ is surjective.  Therefore, we have found a set $\mathcal J_V^{g,j}$ of second category in $\mathcal J_V^j$ such that the linearization of $\overline \partial_{i,J}$ at u is surjective at all elements of $\mathcal J_V^{g,j}$.  

If however $d_{\mathfrak V}<0$, then this kernel is generically empty.  This implies the rigidity of the embedding $u$ of $V$.

\end{proof}

In the statement of our result in Lemma \ref{genericV}, we fix an embedding $u$ of the hypersurface $V$.  This is not quite precise, as we are actually fixing the equivalence class of $u$ in $\mathcal U$ under orbits of the action of Diff$(\Sigma)$.  However, given any two embeddings $u:(\Sigma,i)\rightarrow (X,J)$ and $v:(\Sigma,i)\rightarrow (X,\tilde J)$  of $V$ for $J,\tilde J\in\mathcal J_V^j$, there exists a $\phi\in$ Diff$(\Sigma)$ such that $u=v\circ \phi$.  Thus, a change of embedding $u$ will not affect the outcome of Lemma \ref{genericV}.

For every almost complex structure $j$ on $V$ the previous results provide the following:
\begin{enumerate}
\item A set $\mathcal J_V^{g,j}$ of second category in $\mathcal J_V^j$ with the property that the linearization of the operator $\overline\partial$ at a fixed $j$-holomorphic embedding of $V$  is surjective ($d_{\mathfrak V}\ge 0$) or is injective ($d_{\mathfrak V}<0$).
\item Up to  a map $\phi\in$ Diff$(\Sigma)$,  there is a unique $j$-holomorphic embedding of $V$ for all $J\in \mathcal J_V^j$.
\end{enumerate}

Therefore, consider the following set:
\[
\mathcal J_V^g=\bigcup_{j} \mathcal J_V^{g,j}\subset \bigcup_{j} \mathcal J_V^j=\mathcal J_V.
\]

Note that $\mathcal J_V^g$ is actually a disjoint union of sets.  The following properties hold:  
\begin{enumerate}
\item The set $\mathcal J_V^g$ is dense in $\mathcal J_V$.
\item The linearization of the operator $\overline\partial$ at a fixed $j$-holomorphic embedding of $V$  is surjective ($d_{\mathfrak V}\ge 0$) or is injective ($d_{\mathfrak V}<0$) for any $J\in \mathcal J_V^g$.
\item Up to  a map $\phi\in$ Diff$(\Sigma)$,  there is a unique $j$-holomorphic embedding of $V$.
\end{enumerate}
We can now state the final result concerning genericity that we will need:

\begin{lemma}\label{jV}  
\begin{enumerate}
\item $d_{\mathfrak V}\ge 0$:  Let $\mathcal J_{\mathfrak V}$ be the subset of pairs $(J,\mathcal I_{\mathfrak V})$ which are non-degenerate for the class $\mathfrak V$ in the sense of Def. \ref{nondeg}.  Then $\mathcal J_{\mathfrak V}$ is dense in $\mathfrak I$.
\item $d_{\mathfrak V}<0$: There exists a dense set $\mathcal J_{\mathfrak V}\subset \mathcal J_V$ such that $V$ is rigid, i.e. there exist no pseudoholomorphic deformations of $V$ and there are no other pseudoholomorphic maps in class $\mathfrak V$. 
\end{enumerate}
\end{lemma}

\begin{proof}To begin, we will replace the set $\mathcal J_V$ by $\mathcal J^g_V$ which is a dense subset, as seen from the previous remarks.  Further, for any $(J,\mathcal I_{\mathfrak V})$, $J\in\mathcal J_V^g$, we have surjectivity or injectivity of the linearization at the embedding of $V$.

Consider the case $d_{\mathfrak V}\ge 0$.  Fix a $j$ on $V$.  Then consider the set $\mathcal J_V^{g,j}$ provided by Lemma \ref{genericV}.  The linearization at the embedding of $V$ is surjective for any $J\in \mathcal J_V^{g,j}$.  For any element $(i,u,J,\mathcal I_{\mathfrak V})$ of $\mathcal U$ with $u(\Sigma)\not\subset V$ representing the class $\mathfrak V$ and $J\in \mathcal J_V^{g,j}$, arguments as in the proof of Lemma \ref{genericA} provide the necessary surjectivity.  Therefore, there exists a further set $\mathcal J_{\mathfrak V}^{g,j}$ of second category in $\mathcal J_V^{g,j}\times\{$ initial data$\}$ such that any pair $(J,\mathcal I_{\mathfrak V})\in\mathcal J_{\mathfrak V}^{g,j}$ is nondegenerate, i.e. any $J$-holomorphic curve $u(\Sigma)$ representing $\mathfrak V$ is non-degenerate in the sense of Def. \ref{nondeg}.  

Define $\mathcal J_{\mathfrak V}=\bigcup_{j}\mathcal J_{\mathfrak V}^{g,j}$.  This is a dense subset of $\mathcal J_V^{g,j}\times\{$ initial data$\}$ such that any pair $(J,\mathcal I_{\mathfrak V})\in\mathcal J_{\mathfrak V}$ is nondegenerate. 

If $d_{\mathfrak V}<0$, then restrict to $ \mathcal J^g_V$ as well.  Thereby we have already ensured that $V$ is rigid.  Now apply the proof of Lemma \ref{genericA} to the universal model $\mathcal U$, which we modify to allow only maps $u:(\Sigma,i)\rightarrow (X,J)$ such that $u(\Sigma)\not\subset V$.  Then we can find a set $\mathcal J_{\mathfrak V}$ of second category in  $ \mathcal J^g_V$ such that there exist no maps in class $\mathfrak V$ other than the embedding of $V$.
\end{proof}

Based on this result, we make the following definition:

\begin{definition}
A symplectic hypersurface $V$ is called stable if $d_{\mathfrak V}\ge 0$.
\end{definition}

The following positivity result holds for stable hypersurfaces:

\begin{lemma}\label{positivity}
Let $V$ and $W$ be two symplectic submanifolds of $X$ such that $\mathfrak V,\;\mathfrak W\in H_2(X)$ satisfy $d_\mathfrak V\ge 0$ and $d_\mathfrak W\ge 0$.  Then $\mathfrak V\cdot \mathfrak W\ge 0$.

\end{lemma}

\begin{proof}
By assumption, there exist symplectic submanifolds representing the classes $\mathfrak V$ and $\mathfrak W$.  Taubes' results state that for a Baire set of almost complex structures, the space of submanifolds in each class is non-empty.  A Baire set is a countable intersection of open and dense sets.  Thus, intersecting the two Baire sets for $\mathfrak V$ and $\mathfrak W$, we obtain again a countable intersection of open and dense sets, in particular we obtain a dense set.  Choosing a $J$ from this set, we have $J$-holomorphic embedded representatives of $\mathfrak V$ and $\mathfrak W$ which must intersect non-negatively.  Hence $\mathfrak V\cdot \mathfrak W\ge 0$.

\end{proof}

If one of the two manifolds is non-stable, this result no longer holds.  Consider a $-4$-sphere obtained by blowing up a point in $X$ to obtain an exceptional sphere $e_1$ and then blowing up three distinct points on this sphere to obtain a $-4$-sphere $V$.  This procedure leads to $[e_1]\cdot \mathfrak V=-1$.

\section{\label{structure}The Structure of $\mathcal K_V(J,\mathcal I_A)$}

The goal of this Section is to show that for symplectic hypersurfaces $V$ the spaces $\mathcal K_V(J,\mathcal I_A)$ are smooth, finite, compact spaces which behave well under perturbations of $J$ and the initial data.  These results will provide the foundation for the proof of Prop. \ref{mainprop}.  We begin with the smoothness of $\mathcal K_V(J,\mathcal I_A)$, this will follow almost directly from the results of the previous Section.

To show compactness, we will analyse the behavior of limit curves.  We will show that for generic $(J,\mathcal I_A)$ the limit curve is always a smooth non-multiply covered embedded symplectic submanifold, with the possible exception of the multiply toroidal case.  That case will be addressed separately at the end of this Section.

\subsection{Smoothness}The results of the previous section allow us to prove the following :

\begin{lemma}Fix $A\in H_2(X)$ not multiply toroidal and a class $[\mathcal I_A]$.  There is a set of second category $U\subset \mathcal J_V\times [\mathcal I_A]$  such that for any pair $(J,\mathcal I_A)\in U$ the preimage under the projection 
\[
(i,u,J,\mathcal I_A)\rightarrow (J,\mathcal I_A)
\]
 contains only simple embeddings ($\mathcal E=\emptyset$).
\end{lemma}

\begin{proof}
This follows from a dimension count on the index of the associated operator.
\end{proof}

Moreover, a similar dimension count as well as the Sard-Smale Theorem applied to the projection $(i,u,J,\mathcal I_A)\rightarrow (J,\mathcal I_A)$ proves

\begin{lemma}\label{smoothK}Fix $A\in H_2(X)$ and a class $[\mathcal I_A]$.  There is a set of second category $U\subset \mathcal J_V\times [\mathcal I_A]$ with the following properties:  When a pair $(J,\mathcal I_A)$ is chosen from $U$, then

\begin{enumerate}
\item  $\mathcal K_V(J,\mathcal I_A)$ is empty if $d_A<0$ and $A\ne \mathfrak V$. 
\item If $\mathcal I_A$ is proper, then $\mathcal K_V(J,\mathcal I_A)$ is a smooth 0-dimensional manifold and each point is non-degenerate. 
\item  Assume $A$ is not multiply toroidal.  There is an open neighborhood of pairs in $\mathcal J_V\times [\mathcal I_A]$ such that every pair therein obeys the previous assertions and the number of points in $\mathcal K$ is invariant in this neighborhood.
\end{enumerate}
\end{lemma}

{\bf Remark:} In particular, this result also holds for the class $\mathfrak V$.  Similar results have been proven by Jabuka, see \cite{J}.

Due to this result, we will from now on {\bf assume that $d_A\ge 0$.}

\subsection{Compactness}
We would now like to show, that every sequence of submanifolds
\[
\{C_m, J_m, ({\mathcal I_A})_m\}
\]
 with $({\mathcal I_A})_m$ a proper set of initial data in a fixed class $[\mathcal I_A]$ for all $m$, $J_m\in\mathcal J_V$ and   $C_m$ a point in the corresponding $\mathcal K$ for every $m$ has a subsequence that converges to a $J$-holomorphic submanifold $C$ in class $A$ provided that the limit point of $(J_m,(\mathcal I_{A})_m)$, denoted by $ (J,{\mathcal I_A})$,  is chosen from a suitable Baire set in ${\mathcal J_V}\times [\mathcal I_A]$.

We consider first some generic results  concerning the behavior of the limit curve obtained by Gromov convergence.   Consider a pseudoholomorphic curve $C=\cup C_i$ composed of embedded submanifolds subject to the restrictions imposed by a fixed set of initial data $\mathcal I_A$.  Assume for the moment, that we have a $J$-holomorphic map $f:\Sigma\rightarrow X$ representing the class $A$ such that its image is $C$.  To any multiply covered component we assign its multiplicity $m_i$ and replace the map $f:\Sigma_i\rightarrow X$ by a simple map $\phi_i:\Sigma_i\rightarrow X$ with the same image.  The same is done to any two components with the same image.  The result is a collection of tuples $\{(\phi_i,\Sigma_i,m_i)\}$ with the same image as $f$.  Moreover, we can replace the pair $(\phi_i,\Sigma_i)$ by its image $C_i$.  Furthermore, denoting $A_i=[C_i]$, we obtain $A=\sum_i m_i A_i$.  We must allow for the possibility, that one of  the $A_i=\mathfrak V$.  We therefore consider the decomposition $A=\sum_i m_i A_i+m\mathfrak V$, now assuming that $A_i\ne \mathfrak V$ for all $i$.

%
%
%
%
If we are in the case $\mathfrak V^2<0$ and $d_\mathfrak V\ge 0$, then $\mathfrak V^2=-1$, $d_\mathfrak V=0$ and $g(\mathfrak V)=0$.  This follows from standard arguments:

\begin{lemma}\label{neg}Let $B\in H_2(X)$ with $B^2<0$ and $d_B\ge 0$.  Then for generic $J\in\mathcal J_V$ there exist no embedded irreducible curves in class $B$ unless $B^2=-1$ and $g=0$.
\end{lemma}

\begin{proof}

Assume $d_B\ge0$.  This implies $ K_\omega\cdot B\le B^2<0$.  Using the adjunction formula to determine $K_\omega\cdot B$ provides the estimate $B^2\ge K_\omega\cdot B\ge 2g-2-B^2$ which leads to $g=0$.  By assumption $0>B^2$, hence $B^2=K_\omega\cdot B=-1$.  The result now follows from Lemma \ref{smoothK}.
\end{proof}

As discussed in Section \ref{mapinto}, the limit curve $C$ may have components lying in $V$.  However, by Lemma \ref{index}, we can generically avoid such curves.  This, together with Lemma \ref{jV}, restricts the appearance of curves in class $m\mathfrak V$, in particular in the non-stable case.  This issue will be considered in detail in the following. 

Let us consider briefly the case $A=\mathfrak V$ and $d_\mathfrak V<0$:  Then the results in Section \ref{generic} show, that for generic almost complex structures, the only curve in class $A$ is $V$.  Therefore, for generic $J\in\mathcal J_V$, $\mathcal K_V(J,\mathcal I_A=\emptyset)=\emptyset$ by Lemma \ref{index}.

Due to the remark preceding Lemma \ref{neg} we consider the following two cases: 
\begin{itemize}
\item $d_\mathfrak V\ge 0$ and $\mathfrak V^2\ge 0$ and
\item  $d_\mathfrak V< 0$ or $\mathfrak V^2=-1, d_\mathfrak V=0$.
\end{itemize}

{ $\bf d_\mathfrak V\ge 0$ \bf and $\bf \mathfrak V^2\ge 0$.} Consider a sequence $\{C_m,J_m,({\mathcal I_A})_m\}$.  Then Gromov compactness gives us a finite set of data $\{(\varphi_i,\Sigma_i,m_i,\mathcal I_i)\}$ with $\Sigma_i$ a connected compact Riemann surface, $\varphi_i$ a $J$-holomorphic map from $\Sigma_i$ to $X$ which is an embedding off of a finite set of points which may map into $V$ or contact $V$ in accordance with the data given by ${\mathcal I_i}$ and $m_i\in \mathbb N$.  We choose $(J,{\mathcal I_A})$ from a Baire set as discussed in the previous section, such that $\varphi_i(\Sigma_i)\cap \varphi_j(\Sigma_j)$ is finite for $i\ne j$ in accordance with the intersection product on homology for the classes $A_j,A_i$.  Furthermore, denoting the push forward of the fundamental class of $\Sigma_i$  by $A_i$, we obtain $A=\sum m_iA_i$ and the image $\cup \varphi_i(\Sigma_i)$ is connected and contacts all the data in ${\mathcal I_A}$.  Moreover, $\cup \mathcal I_i=\mathcal I_A$.  This implies that $\sum_id_k^i=d_k$ ($k\in\{1,2\}$) and $\sum_il_k^i=l_k$ ($k\in \{1,2,3\}$).  Further, we may assume that $2d_1^i+d_2^i-l_2^i-2l_3^i\le 2(d_{A_i}-l_{A_i})$ for each set $\mathcal I_i$.  This condition simply states that the dimension of the moduli space is larger than the degrees of the insertions, thus guaranteeing that curves exist.  An inequality in the opposite direction would provide too many constraints on curves in class $A_i$, thus effectively ruling out the existence of such a curve for generic $J$.

Consider any two pairs of points on $V$ with a prescribed contact order given in the initial data.  These must stay separate in the limit, as all of the initial data $\{\mathcal I_A\}_m$ lies in the same proper class for all $m$ as does the limit set $\mathcal I_A$.  Moreover, the components they lie in cannot limit to a multiple cover of the same curve, otherwise this component would need to meet more points than given by $l_{A_i}$.  In particular, this implies that $\sum_il_{A_i}\ge l_A$.

The properness of the initial data $\mathcal I_A$ as well as our estimate on the data in $\mathcal I_i$ allows the first estimate, the fact that $\cup \mathcal I_i=\mathcal I_A$ and $\sum_i l_{A_i}\ge l_A$ the final equality:
\begin{equation}\label{dai_ge_da}
2\sum  d_{A_i} \ge  \sum 2d_1^i+d_2^i-l_2^i-2l_3^i+2l_{A_i}= 2d_A.
\end{equation}
We have thus shown, that there is a Baire set of pairs $(J,\mathcal I_A)$ such that 
\begin{equation}
  \sum_i d_{A_i}\ge d_A.  
\end{equation}
On the other hand, the following Lemma (primarily proven by Taubes \cite{T}, see also \cite{DL}) holds, we provide a proof suited to the situation at hand:

\begin{lemma}\label{conv}For generic pairs $(J,{\mathcal I_A})$, either $\sum d_{A_i} < d_A$ or one of the following hold:

\begin{enumerate}
\item $\{C_m\}$ has a subsequence which converges to a $J$-holomorphic submanifold $C$ with fundamental class $A$.  Moreover, the limit curve $C$ intersects $V$ locally positively and transversely.  
\item $A$ is multiply toroidal.  Furthermore, the data given by Gromov convergence consists of one triple $(\varphi_1,\Sigma_1,m_1)$ where $\Sigma_1$ is a torus, $\varphi_i$ embeds $\Sigma_1$ and $A=m_1A_1$.  This includes the case $A=m\mathfrak V$ if $V$ is square 0 torus.
\item Except possibly if $A=m\mathfrak V$ ($m\ge 1$) and $V$ is a square 0 torus,  $C\not\subset V$.
\end{enumerate}
\end{lemma}

\begin{proof}We decompose the class $A$ as before, however distinguishing two types of classes as follows:  Let $B_i$ denote components with negative square, $A_i$ components with nonnegative square.  Then write $A=m\mathfrak V +\sum m_iB_i+r_iA_i$.  In the following we will allow the case $A=m\mathfrak V$.  

Lemma \ref{neg} ensures, that we can find a generic set of almost complex structures such that $B_i^2=-1$ and $g(B_i)=0$ for all $i$.  This in particular ensures, that $B_i\ne B_j$ for $i\ne j$ due to positivity of intersections.  Moreover, if $A\cdot B_i<0$ for some $i$, then we may assume that this holds for $C_m$ with sufficiently large $m$ in the sequence $\{C_m, J_m, ({\mathcal I_A})_m\}$.  However, each $C_m$ is a connected $J_m$-holomorphic submanifold representing $A$, hence  again by positivity of intersections, $A\cdot B_i<0$ can only occur if $A=m_i B_i$ for some $i$.  In particular, $A\ne\mathfrak V$.  Again, by Lemma \ref{neg}, it follows that $m_i=1$ generically.  Thus $A$ would be represented by an embedded submanifold and the proof is done.

Assume in the following, that $A\cdot B_i\ge 0$ for all $i$.  Lemma \ref{smoothK} shows, that we can find a generic set of pairs $(J,\mathcal I_A)$ such that $d_{A_i}\ge 0$ for  each curve in class $A_i$, the same holds for $\mathfrak V$ by assumption.  

Further, if $d_{A_i}\ge 0$ and $A_i^2\ge 0$, then $d_{r_iA_i}\ge 0$ for any positive integer $r_i$:
\[
0\le 2d_{A_i}\le 2r_id_{A_i}=-K_{\omega}\cdot(r_iA_i)+r_iA_i\cdot A_i\le
\]
\[
\le -K_{\omega}\cdot (r_iA_i)+r_i^2A_i\cdot A_i=2d_{r_iA_i}.
\]
Note that this holds in particular for $m{\mathfrak V}$.

For such a generic choice of $(J,\mathcal I_A)$, let $C$ be a connected curve representing $A$, which meets the initial data $\mathcal I_A$.  Then
\[
2d_A=-K_{\omega}\cdot(m{\mathfrak V})+\sum_i-K_\omega\cdot(m_iB_i)+\sum_i-K_{\omega}\cdot(r_iA_i)+
\]
\[
+m^2{\mathfrak V}^2+\sum_im_i^2B_i^2+\sum_ir_i^2A_i^2+2\sum_im{\mathfrak V}\cdot m_iB_i+2\sum_im{\mathfrak V}\cdot r_iA_i+
\]
\[
+2\sum_{i> j}m_im_jB_i\cdot B_j+2\sum_{i> j}m_ir_jB_i\cdot A_j+2\sum_{i> j}r_ir_jA_i\cdot A_j
\]
\[
\ge 2md_{\mathfrak V}+ (m^2-m){\mathfrak V}^2+2\sum_i r_id_{A_i}+2\sum_im{\mathfrak V}\cdot r_iA_i+2\sum_{i> j}r_ir_jA_i\cdot A_j+
\]
\[
+\sum_i (m_i^2-m_i)B_i^2+2\sum_{i> j}m_im_jB_i\cdot B_j+2\sum_{i> j}m_ir_jB_i\cdot A_j+2\sum_im{\mathfrak V}\cdot m_iB_i.
\]
Consider the terms in the last line.  Recalling that $B_i^2=-1$, they can be rewritten as 
\[
\sum_i 2m_iA\cdot B_i-2m_i^2B_i^2+(m_i^2-m_i)B_i^2\;=\;\sum_i 2m_iA\cdot B_i+m_i^2+m_i\;\ge 0
\]
and thus we obtain the estimate
\[
2d_A\ge 2d_{\mathfrak V}+2\sum_i d_{A_i}.
\]
Hence either $d_A> d_{\mathfrak V}+\sum_i d_{A_i}$ or  the following hold:
\begin{itemize}
\item $m_i=0$ for all $i$, i.e. there are no components of negative square,
\item $A_i\cdot A_j=0=A_i\cdot {\mathfrak V}$ for $i\ne j$,
\item  $m=1$ or $d_{\mathfrak V}=0$ and ${\mathfrak V}^2=0$ and
\item $r_i=1$ or $d_{A_i}=0$ and $A_i^2=0$.
\end{itemize}

The limit curve is connected as we started with a connected curve.  The second result shows that $A=rA_1$ or $A=m\mathfrak V$.  The last two refine this to show that the curve $C$ representing $A$ is an embedded $J$-holomorphic submanifold with a single non-multiply covered connected component meeting the initial data $\mathcal I_A$ with $J\in \mathcal J_V$ or $d_A=0$.  

Thus we are done if $A\ne {\mathfrak V}$ and $d_A>0$.  Now consider the following cases:
\begin{enumerate}

\item $A\ne {\mathfrak V}$ and $d_A=0$:  The results above show that this implies either $A=m\mathfrak V$ with $m\ge 2$ and $\mathfrak V^2=0$, i.e. $A$ is multiply toroidal or that $A\ne m\mathfrak V$ is (multiply) toroidal or $A$ is represented by an embedded $J$-holomorphic submanifold as stated above.  In the latter two cases we have $C\not\subset V$.

\item $A={\mathfrak V}$:  The limit curve is clearly an embedded curve under all circumstances, only its placement relative to $V$ is an issue.  If $C\ne V$, then we are done.  If $C=V$, then either a generic choice of $\mathcal I_A$ will prevent this limit from occurring due to dimension reasons, see Lemma \ref{index}, or $V$ is a square 0 torus.

\end{enumerate}

If $A\cdot\mathfrak V>0$, we can perturb $C$ to be transverse to $V$, see \cite{M4}, \cite{M5}.  If $A\cdot\mathfrak V=0$, then either the curves do not meet or we are in the toroidal case again.

\end {proof}

\subsection{Non-stable Hypersurfaces and Exceptional Curves} Assume that $V$ is non-stable or that $V$ is an exceptional sphere.  We would like to argue as in the stable case.  However, now the case $A_i=\mathfrak V$ must be given separate consideration.  Lemma \ref{jV} resp. results on exceptional curves show,  that  we can find a generic set of almost complex structures, such that $V$ is rigid and there are no other curves in class $\mathfrak V$.  In the following, we choose only complex structures from this set.  

As before, we can consider a sequence $\{C_m,J_m,(\mathcal I_A)_m\}$ and obtain the same result as above from Gromov compactness with one exception:  The case $A_i=\mathfrak V$ with $m_i\ge 1$ must be considered closer:  Even though we are working in the case $d_\mathfrak V<0$, it is possible for a multiple class $m\mathfrak V$ to have $d_{m\mathfrak V}\ge 0$.  For this reason, we will distinguish the following two objects:
\begin{enumerate}
\item Classes $A_i=\mathfrak V$ with $m_i>1$ which correspond to components of the curve $C$ in class $m\mathfrak V$, but which are NOT multiple covers of a submanifold in class $\mathfrak V$.  If $\mathfrak V^2< 0$, then positivity of intersections shows that any curve $C$ can contain at most one component in class $m\mathfrak V$  for all $m$ and this component must coincide with the manifold $V$ (It could be a multiple cover of course.).  This situation was studied in greater generality in \cite{Bi}.  Furthermore, if  $\mathfrak V^2\ge 0$ and a class $A_i=m\mathfrak V$ occurs, then the results of Lemma \ref{genericA} apply.  We may therefore assume , that $A_i^2\ge 0$ in the following.
\item The specific "class" $mV$ which corresponds to components which have as their image the hypersurface $V$.  (Of course, the homology class associated to this "class" is $m\mathfrak V$.  We wish to emphasize the distinction between the geometric object associated to this "class" and the previous one.)  This can only occur if $d_{m\mathfrak V}<0$.
\end{enumerate}
Note further, that we can choose our almost complex structures such that the components corresponding to $mV$ are rigid, while those in $m_i\mathfrak V$ are not.  Such a decomposition is not necessary in the case $d_\mathfrak V\ge 0$, as $V$ acts no differently than any other curve in the class $\mathfrak V$.  In the current situation, the specific hypersurface $V$ is singled out in the class, while all others can be excluded.

Lemma \ref{index} allows us to exclude curves of the second type for generic initial data.  

Moreover, if $V$ is an exceptional sphere, $d_{m\mathfrak V}<0$ if $m\ge 2$.  Hence this rules out any classes of first type for exceptional curves.  Any appearance of the class $\mathfrak V$ in a decomposition of $A$ must therefore stem from a cover of $V$ as there is but one representative of $\mathfrak V$.  

Hence, in all cases the calculations from the previous section apply.  We again obtain an embedded curve or $A$ is multiply toroidal.  This proves compactness of the relative space $\mathcal R_V(J,\mathcal I_A)$ in this case.

The compactness results from this section are summarized in the following Lemma:
   
\begin{lemma}\label{finiteK}Fix a symplectic hypersurface $V$, a class $A\in H_2(X)$ and assume $A$ is not (multiply) toroidal.  Then there exists a set of second category
$U\subset \mathcal J_V\times [\mathcal I_A]$ with the following properties:  When a pair $(J,\mathcal I_A)$ is chosen from $U$, then

\begin{enumerate}
\item $\mathcal K_V(J,\mathcal I_A)$ is a finite collection of points.  
\item Any submanifold in $\mathcal K_V(J,\mathcal I_A)$ meets $V$ locally positively and transversely.  
\item $C\not\subset V$  
\item Assertion 1 is an open condition in $J_V\times [\mathcal I_A]$.
\end{enumerate}
\end{lemma}

\begin{proof}Assertions (2) and (3) are proven in the previous Section by direct calculation.  Moreover, these calculations show also that with the exception of multiply toroidal classes, the spaces $\mathcal K_V$ are compact.  Together with Lemma \ref{smoothK} it follows that $\mathcal K$ is finite.  Furthermore, a direct application of the implicit function theorem gives part 2 of the Lemma.

\end{proof}

\section{\label{relinv}The relative Ruan Invariant}

This section is devoted to the precise definition of the relative Ruan invariant.  We first describe how to define a number associated to the spaces $\mathcal K_V(A,J,\mathcal I_A)$.  In the following, we show that this is a deformation invariant of the symplectic structure $\omega$ on $X$.

Let $V$ be a stable symplectic hypersurface.  The constructions in this section again assume that no curves can limit into the fixed hypersurface $V$.  Moreover, throughout we assume that $A$ is not multiply toroidal.  

\subsection{The number $Ru^V(A,[\mathcal I_A])$}

For a fixed $A\in H_2(X)$ which is not multiply toroidal, consider the set $\mathcal K_V(J,\mathcal I_A)$ for generic pairs $(J,\mathcal I_A)$ and a fixed class $[\mathcal I_A]$.  Then define the number
\begin{equation}
Ru^V(A,[\mathcal I_A])=\sum_{C\in \mathcal K_V(J,\mathcal I_A)}r(C,\mathcal I_A).
\end{equation}

\subsection{The Definition of $r(C,\mathcal I)$}

This number is determined through an analysis of the behavior of the operator $D$ (Eq. \ref{D}) under perturbation to a $\mathbb C$-linear operator.  The methods used in this Section can be found in Kato \cite{K} (Chaps. II and VII), McDuff-Salamon \cite{McS} (Appendix A) and the papers of Taubes in \cite{T}.  To fully understand this, consider the set $\mathcal F_{\mathbb R}$ of real linear Fredholm operators from a Banach space $X$ to a Banach space $Y$ of index n.  This space can be decomposed into components $\mathcal F_{\mathbb R}^k$ consisting of those operators with kernel of dimension k.  The minimal codimension of $\mathcal F_{\mathbb R}^{k+1}$ in $\mathcal F_{\mathbb R}^k$ is given by 
\begin{equation}\label{codim}
n+2k+1.
\end{equation} 

Consider now a real analytic perturbation of $D$.  More precisely, consider a path of operators $A_t:[0,1]\rightarrow \mathcal F_{\mathbb R}$ with the following properties:
\begin{enumerate}
\item $A_t$ depends real-analytically on the parameter $t$,
\item $A_0=D$,
\item $A_t-A_0$ is a bounded $0$th order deformation of $D$ and 
\item $A_1$ is $\mathbb C$-linear.
\end{enumerate}
The third condition ensures that we stay in $\mathcal F_{\mathbb R}$, more precisely we even stay within the set of elliptic operators, as $D$ is elliptic.  Moreover, each of the operators has compact resolvent.  The operator $D$ is not $\mathbb C-$ but $\mathbb R-$linear, thus we view it formally as a map between the underlying real bundles.  However, each of these bundles carries a holomorphic structure, hence we can consider an analytical extension of this path.  This can be achieved by choosing a Sobolev completion of the domain and the range of the operator $D$ making $D$ a bounded operator.  Then we can extend this real analytical perturbation to an analytical perturbation over a domain $U\subset \mathbb C$ containing $[0,1]$ in its interior.  This can be done such that the analytical perturbation preserves the third condition above.  This is now a perturbation in $\mathcal F_{\mathbb C}$.  Applying the results in Kato, in particular Sections II.1 and VII.1, we conclude that for the real analytic path $A_t$ the following hold:
\begin{enumerate}
\item Either the kernel of $A_t$ is nonempty for all $t$ or it is nonempty for at most finitely many $t$.  This result is relevant in the case $n=0$ and $k=0$:  The path $A_t$ will intersect the component with kernel of dimension $\ge 1$ transversally.  Moreover, the dimension count \ref{codim} shows that we can choose generic perturbations which intersect only the component with $k=1$ but not any components with $k\ge 2$.  Note that $A_1$ is $\mathbb C$-linear, hence the kernel at $t=1$ cannot have dimension 1.
\item The dimension of the kernel of $A_t$ can only change at a finite number of $t$.  If $n\ge 1$ or $k\ge 1$, then we can again choose a generic perturbation which will not intersect any of the higher codimension components.
\end{enumerate}

Hence, we can choose a generic real analytical path $A_t$ such that for $(n,k)=(0,0)$ we have only a finite number of $t$ at which the kernel has dimension 1 and for $n\ge 1$ or $k\ge 1$ we can ensure that the dimension of the kernel is preserved along the whole path over $[0,1]$.

We shall always assume that we have chosen an almost complex structure $J$ such that $D$ has trivial cokernel.  Hence we consider only $n=k$ in the following.  This implies that for a fixed value $n$, objects with a larger kernel will be of codimension $3k+1$ at least.  This allows the following constructions:  
\begin{itemize}
\item If $n>0$, any continuous path connecting $D=A_0$ to a $\mathbb C$-linear operator and a generic real analytical path $A_t$ as described above bound a disk such that every operator in the disk has the same size kernel and cokernel.   In particular, we can use the kernel of $A_1$, which as a $\mathbb C$-linear operator carries a natural orientation induced by $J$, to uniformly orient all of the kernels in the disk.   Moreover, as this orientation is defined by the almost complex structure $J$, we can choose any generic $\mathbb C$-linear operator and obtain the same orientation of $\ker(D)$. 

\item If $n=0$, then two paths can differ by the number of crossing points of the codimension 1 stratum $\mathcal F_{\mathbb R}^1$.  However, each curve connecting $D$ and $A_1$ must have the same number of crossings mod 2 for a generically chosen $A_1$.  Moreover, the almost complex structure orients the 0-dimensional kernel of any $\mathbb C$-linear operator, this orientation must be equivalent and is given by associating $\pm 1$ to each point in the kernel.  Hence any path connecting two $\mathbb C$-linear operators $A_1$ and $A_1'$ must cross the stratum $\mathcal F_{\mathbb R}^1$ an even number of times, i.e. any point which has its orientation reversed must have it reversed again.  Thus the number of crossings mod 2 is generically independent of the choice of $A_1$.  This defines the spectral flow mod 2 for a real analytical path:  Let $N$ be the number of crossing points, then the spectral flow mod 2 for the path $A_t$ is $(-1)^N$.  

Now consider any continuous path with $N$ number of crossings.  Then as in the case $n>0$, any real analytical path $A_t$ with $N$ crossings of the same stratum bound a disk.  Thus the spectral flow can be computed for a generic continuous path.

We have shown that this number is independent of the generic choice of continuous path and endpoint.

\end{itemize}

This is the general setup for our definition of $r(C,\mathcal I)$.  We assume in the following discussion, that $C\ne V$.

{$\bf \; d_A=0$:}  Choose the $\mathbb C$-linear operator $A_1$ such that it has trivial kernel and cokernel.  This can be achieved generically because $A_1$ is $\mathbb C$-linear and the dimension count \ref{codim}.  Now define
\begin{equation}
r(C,\mathcal I)=(-1)^N
\end{equation}
for a generic continuous path $A_t$ connecting $D$ and $A_1$.  Note in particular, that we can use the path to define an orientation on the kernel of $D$ determined by the almost complex structure.

{$\bf \;d_A>0$}
To each point in the initial data $\mathcal I_A$ we will assign a space as follows:
\begin{itemize}
\item $\Omega_{d_1}$:  To each point $z\in \Omega_{d_1}$ associate the fiber $N\vert_z$ of the normal bundle $N$ of the curve $C$.  This defines a direct sum $\mathfrak E_{d_1}=\oplus_{z\in \Omega_{d_1}}N\vert_z$.
\item $\Gamma_{d_2}$:  Consider the intersection point of the curve $C$ and an element $\gamma\in\Gamma_{d_2}$.  Through a slight perturbation of $\gamma$, we may assume that $\gamma$ intersects $C$ in such a manner, that the quotient $N\vert_z/ T\gamma$ is well defined.  In other words, by perturbing slightly, we ensure that $T\gamma\vert_z$ is a line in the normal bundle fiber over $z$.  Then associate to each $\gamma$ the space $\mathfrak g_{\gamma}=N\vert_z/ T\gamma$.  Each $\mathfrak g_\gamma$ is oriented, hence the space $\mathfrak E_{d_2}=\oplus_{\gamma\in\Gamma_{d_2}}\mathfrak g_{\gamma}$ is an oriented ordered direct sum of oriented lines.
\item $\Omega_{l_1}$:  This set consists of pairs, as will all of the following sets.  They consist of a geometric datum and an intersection order.  Recall the definition of the map $G_k$ given by Def \ref{contactmap}.  To each pair $(z,s)$ associate the space $N\vert_z\otimes S_s$.  This space is again oriented.  Hence $\mathfrak E_{l_1}=\oplus_{(z,s)\in \Omega_{l_1}}N\vert_z\otimes S_s$ is oriented.
\item $\Gamma_{l_2}$:  Assign the space $\mathfrak E_{l_2}=\oplus_{(\gamma,s)\in\Gamma_{l_2}}\mathfrak g_{\gamma}\otimes S_s$.
\item $\Upsilon_{l_3}$:  In this set, we make no restrictions on the contact location with $V$.  Note that $S_k$ not only encodes contact order, but also encodes a location in the hypersurface $V$, see the discussion in Section \ref{contactorder}.  Hence we define the space $ \mathfrak E_{l_3}=\oplus_{(V,s)\in\Upsilon_{l_3}} S_s$.  The corresponding evaluation map is similar to the one constructed in Section \ref{contactorder}, we do not specify a fiber $F_z^C$.  

\end{itemize}

Consider the linear map $H:\ker(D)\rightarrow \mathfrak E_{d_1}\oplus \mathfrak E_{d_2}\oplus \mathfrak E_{l_1}\oplus \mathfrak E_{l_2}\oplus \mathfrak E_{l_3}$, which is composed of evaluation maps and maps $G_k$ defined in Def. \ref{contactmap}.  Choose a generic continuous path $A_t$ and use this path to orient $\ker(D)$ as described above.  Then $H$ is a map between oriented vector spaces, moreover our calculations in the previous Sections show that this map is an isomorphism for suitably generic $(J,\mathcal I)$.  Define
\begin{equation}
r(C,\mathcal I)=\mbox{ sign(det(}H)).
\end{equation}

\begin{lemma}
For generic $(J,\mathcal I)\in \mathcal J_V\times [\mathcal I]$ the number $r(C,\mathcal I)$ is well-defined when $C\in \mathcal K_V(J,\mathcal I)$ .
\end{lemma}

\begin{proof} We need to show that the linear map has a  determinant with a well-defined sign for generic $(J,\mathcal I)\in \mathcal J_V\times [\mathcal I]$.  This follows from the genericity results obtained in the previous Sections and the homotopy properties discussed above.

\end{proof}

Note that this definition agrees with Taubes' definition if $l_A=0$ as well as in the case $\mathcal I_A=\emptyset$, albeit with a different underlying set of almost complex structures.

\subsection{Invariant Properties of $Ru^V(A,[{\mathcal I_A}])$}

Given the triple $(X,V,\omega)$ we denote its symplectic isotopy class by $[X,V,\omega]$.  This class contains all triples $(X,\tilde V,\tilde \omega)$ such that there exists a  smooth one parameter family $(X,V_t,\omega_t)$ with
\begin{enumerate}
\item $(X,V_0,\omega_0)=(X,V,\omega)$,
\item $(X,V_1,\omega_1)=(X,\tilde V,\tilde \omega)$,
\item $\omega_t\in \mathcal S_X^{V_t}$ and
\item $[\omega_t]=[\omega]\in H^2(X)$.
\end{enumerate}
The triple $(X,V,\omega)$ is deformation equivalent to $(\tilde X,\tilde V,\tilde \omega)$ if there exists a diffeomorphism $\phi:\tilde X\rightarrow X$ such that $[\tilde X,\phi^{-1}(V),\phi^*(\omega)]=[\tilde X,\tilde V,\tilde \omega]$.

\begin{theorem}\label{invN}
The number $Ru^V(A,[{\mathcal I_A}])$ depends only on the deformation class of $(X,V,\omega)$, the class $A\in H_2(X)$, the initial class $[\mathcal I_A]$ and the ordering of the data in the sets $\Gamma_*$.  In particular, it does not depend on a particular choice of $(J,\mathcal I_A)$.
\end{theorem}

The proof of this Theorem will occupy the rest of this Section.  To begin, we consider the deformation invariance.  Consider a family of symplectic forms $\{\omega_t\}$ parametrized by $[0,1]$.  Let $(J_0,\mathcal I^0_A),(J_1,\mathcal I^1_A)\in \mathcal J_V\times [\mathcal I_A]$ be two pairs associated to $\omega_0$ and $\omega_1$ such that $J_*$ is compatible with $\omega_*$, $(J_*,\mathcal I^*_A)$ is suitably generic in the sense of the previous Sections and $[\mathcal I_A^1]=[\mathcal I_A^0]$.  

\begin{definition}Define the set $\Gamma(\{\omega_t\},[\mathcal I_A])$ to be the space of smooth sections $\mathfrak s = \{(t,J_t,\mathcal I^t_A)\}$ over $[0,1]$ such that 
\begin{enumerate}
\item $J_t$ is $\omega_t$-compatible, 
\item $\mathcal I_A^t\in [\mathcal I_A]$ and
\item  at $t=0,1$ we have the triples $(0,J_0,\mathcal I_A^0)$ and $(1,J_1,\mathcal I_A^1)$ with the associated pairs  chosen above.
\end{enumerate}
\end{definition}

The purpose of the following is to show that the number $Ru^V(A,[{\mathcal I_A}])$ is an invariant of the deformation class of the symplectic structure $\omega$ on $X$.  To prove this, we will extend the universal space $\mathcal U$.  Define a universal space $\mathcal Y$ similar to $\mathcal U$ consisting of Diff$(\Sigma)$ orbits of a 5-tuple $(i,u,t,J,\mathcal I_A)$ with the following additional properties:
\begin{itemize}
\item For $t=0,1$, only the values for $(J,\mathcal I_A)$ chosen initially are allowed.
\item For each corresponding pair of curves $\gamma_0,\gamma_1$ in the respective sets $\Gamma_{d_2}$, fix a smooth cobordism $X_\gamma$ of $\gamma_0$ to $\gamma_1$.  Repeat the same for the relative curves: For each corresponding pair of data $(\gamma_0, s), (\gamma_1,s)$ in the respective sets $\Gamma_{l_2}$, fix a smooth cobordism $X_\gamma$ of $\gamma_0$ to $\gamma_1$.  Require the intersection points of the image of $u$ and the data in $\Gamma^t_{d_2}\cup\Gamma^t_{l_2}$ to lie on the submanifold $\prod_{\gamma\in\Gamma_{d_2}\cup\Gamma_{l_2}}X_\gamma\subset X^{d_2+l_2}$. 

\end{itemize}

For a fixed section $\mathfrak s\in\Gamma(\{\omega_t\},[\mathcal I_A])$, define the space $\Xi_\mathfrak s$ via the pull-back diagram
\[
\begin{diagram}
\node{\Xi_\mathfrak s}\arrow{e}\arrow{s}\node{\mathcal Y}\arrow{s,l}{\pi}\\
\node{\mathfrak s}\arrow{e,b}{T}\node{\mathcal J_V\times [\mathcal I_A]}
\end{diagram}
\]
In other words, $\Xi_\mathfrak s$ is the collection of pairs $(t,C)$ of $t\in[0,1]$ and  submanifolds $C$, such that for each $t$, the submanifold $C\in\mathcal K(\mathfrak s(t))$ and meets the submanifold $\prod_{\gamma\in\Gamma_{d_2}\cup\Gamma_{l_2}}X_\gamma$ as described above.

\begin{lemma}\label{definv}Fix a smooth family $\{\omega_t\}$.  Then there exists a Baire set $U\in \mathcal J_V\times [\mathcal I_A]$ such that any two points in $U$ can be joined by a section $\mathfrak s:[0,1]\rightarrow\Gamma(\{\omega_t\},[\mathcal I_A])$ such that $\Xi_{\mathfrak s}$ is a smooth 1-dimensional manifold.  If $A$ is not multiply toroidal, then $\Xi_{\mathfrak s}$ is compact.

\end{lemma}

The proof of this Lemma will rely on the following result:  Consider the following pull-back diagram

\[
\begin{diagram}
\node{\mathcal X}\arrow{e}\arrow{s}\node{\mathcal Y}\arrow{s,l}{\pi}\\
\node{\Gamma(\{\omega_t\},[\mathcal I_A])}\arrow{e,b}{T}\node{\mathcal J_V\times [\mathcal I_A]}
\end{diagram}
\]
where $T$ is the evaluation map  and $\pi$ is the projection onto the last two components.

\begin{lemma}
The map $\pi$ is transverse to $T$.
\end{lemma} 

\begin{proof}
As in Section \ref{generic}, we will need to distinguish the cases $A\ne\mathfrak V$ and $A=\mathfrak V$.  In the former case, the results of Section \ref{generic} immediately show that the differential of $\pi$ is surjective.  In the latter case, we need to again be wary of $\xi=0$, which is not in the image of $d\pi$.  However, the differential of the evaluation map $T$ can attain this value, hence again transversality is attained.

\end{proof}

We now turn our attention to the proof of Lemma \ref{definv}:

\begin{proof}The previous Lemma ensures that the pull-back space $\mathcal X$ is a smooth manifold.  Consider the map $P:\mathcal X\rightarrow \Gamma(\{\omega_t\},[\mathcal I_A])$.  Applying the Sard-Smale Theorem to this map proves the smoothness of the space $\Xi_{\mathfrak s}$ as well as the claim on the dimension.  The compactness of this space follows from the arguments on compactness in Section \ref{structure}.

\end{proof}

This shows that the spaces of connected submanifolds $\mathcal K_V(J,\mathcal I)$ are invariant under deformation of the symplectic structure and do not depend upon the particular choice of $J$ or $\mathcal I$.  Moreover, it is clear that these spaces depend on the class $[\mathcal I]$ and the ordering of this class.  

Hence, $Ru^V(A,[\mathcal I_A])$ is an invariant of the deformation class of the symplectic structure and otherwise depends on $A$, the hypersurface $V$ and the ordered initial class $[\mathcal I_A]$.

{\bf Remark:} \begin{enumerate}
\item Thus far, the relative Ruan invariant is defined when $A$ is not multiply toroidal.  If we consider multiply toroidal $A$, then Taubes has shown that already in the absolute case it is not possible to define a meaningful invariant allowing only connected curves.  This first step towards an invariant for disconnected curves will be taken in Section \ref{taubesinv}.

\item Motivated by the work by Maulik and Pandharipande (\cite{MP2}), it is natural to ask, whether the relative Ruan invariant depends not on the deformation class of $(X,V,\omega)$, but actually only on the class $\mathfrak V$.  In particular, it would not depend on the precise choice of $V$.  This would have interesting ramifications, see for example the result for K3 surfaces in Thm. \ref{k3vanish}.  The result in \cite{MP2} relies on the induced mapping $H^*(V)\rightarrow H^*(X)$, which in our 4-manifold setting is only interesting on the level of $H^1$.  In particular, if we have no insertions in $H^1$ or if $X$ is simply connected, then this condition would be trivially fulfilled for any representative of $\mathfrak V$ and  hence this map would provide no way to distinguish between different representatives of the class $\mathfrak V$.  Furthermore, the rank of the skew-symmetric part of the restriction of the intersection pairing to the pull back of $H^1(X)$ to $H^1(V)$ was shown to depend only on the class $A$, see \cite{Li3}.  Thus, we are led to ask the following question: 

\begin{question}
On a symplectic 4-manifold, can we find two hypersurfaces $V_1$ and $V_2$ representing the class $\mathfrak V$,  such that the images of  $H^1(V_1)$ and $H^1(V_2)$ in $H^1(X)$ differ?
\end{question}  

\end{enumerate}

\subsection{\label{grtex}Some Calculations}
\subsubsection{Genus 0 Curves}

In this section we consider relative Ruan invariants in genus 0 relative to submanifolds $V$ with $d_\mathfrak V\ge 0$.  Unless otherwise stated, we assume that $X$ is not rational or ruled.  It was shown in \cite{LS} that all relative Gromov-Witten invariants of $(X,\omega)$ in genus 0 vanish if $X$ is minimal.  These invariants are concerned with connected curves representing the class $A$.  Moreover, the proof in \cite{LS} shows, that for curves of genus 0 not only do the GW-invariants vanish, but for generic $(J,\mathcal I_A)$ the spaces underlying the invariants are empty.  Hence all relative Ruan invariants are trivial as well for genus 0 classes if $X$ is minimal.  Moreover, the proof in \cite{LS} shows, that if $X$ is non-minimal, then generically the only possible genus 0 connected curves are embedded $-1$-spheres.  Hence, we consider only classes $A=E$ of an exceptional sphere. 
 
For each exceptional curve class $E$ we can state the following:  For each $J\in \mathcal J_V$ there is a unique symplectic submanifold representing the exceptional curve.  Hence it follows that 
\begin{equation}
\mathcal K_V(E,J,\mathcal I_{E})=\left\{\begin{array}{cc} (E,1,\mathcal I_{E})& \mathfrak V\cdot E>0\\
\emptyset &  \mbox{otherwise}
\end{array}\right.
\end{equation}
and
\begin{equation}
\mathcal I_{E}=\left\{\begin{array}{cc}\Upsilon^{E\cdot\mathfrak V}_V& \mathfrak V\cdot E>0\\
\emptyset & \mbox{otherwise}
\end{array}.\right.
\end{equation}
These results follow from the uniqueness of exceptional curves and Lemma \ref{values}.
  
If $\mathfrak V\cdot E\le 0$, then the relative Ruan invariant vanishes.  In the remaining case, the calculation of $r(E_i,1,\mathcal I_i)$ relies on a path of operators, as described in Section \ref{relinv}, which corresponds to a change in complex structure on the normal bundles.  This path can have only a finite number of points at which the kernel of the operator has dimension greater than 0.  This however would imply the existence of more than one exceptional curve for certain almost complex structures, which we can rule out topologically.  Hence $r(E_i,1,\mathcal I_i)=1$.  It follows that 
\begin{equation}
Ru^V(A,[\mathcal I_A])=1.
\end{equation}
Hence we have proven

\begin{theorem}\label{spheres}
Let $X$ be a non-rational, non-ruled symplectic 4-manifold and $V$ a symplectic hypersurface.  Let $A\cdot \mathfrak V>0$ and the genus of $A$ as determined by the adjunction formula be 0.  Denote the exceptional curves of $X$ by $E_i$.  Then 
\begin{equation}
Ru^V(A,[\mathcal I_A])=\left\{\begin{array}{cc}
1& \mbox{ if }A=E_i,\\
0&\mbox{ otherwise.}
\end{array}\right.
\end{equation}
\end{theorem}

It should be noted, that $\mathfrak V\cdot E<0$ can only occur if $V$ is non-stable, see Lemma \ref{positivity}.

\subsubsection{Algebraic K3-surfaces}

Let $X$ be a K3 surface, i.e. a surface with trivial canonical bundle and $b_1=0$.  In \cite{MP2} the authors consider relative invariants for a K3 surface relative to a quartic in an application of the sum formula for Gromov-Witten invariants.  It is striking, that the only invariants which are non-trivial, are those where the class of the divisor $V$ and the class of the relative curve $A$ are the same.  In the following, we provide some explanation for this behavior.

K3 surfaces have been extensively studied and much is known about their moduli.  We review briefly some facts and introduce notation, details can be found in \cite{BPV}.  

For a K3 surface, the Betti numbers take the values $b_1=0$ and $b_2=22$.  The group $H^2(X,\mathbb Z)$ is an even unimodular lattice with a quadratic form $q$ given by the intersection pairing.  This pairing has signature $(3,19)$ and hence is isomorphic to the pairing $(\cdot,\cdot )$ given by $L=3H\oplus 2(-E_8)$.    Fix a lattice isomorphism $\phi:(H^2(X,\mathbb Z),q)\rightarrow L$, a K3 surface with a fixed choice of $\phi$ is called a marked K3 surface.  The period of a marked K3 surface $X$ is given by  a choice of $[J]\in L\otimes \mathbb C\cong H^2(X,\mathbb C)$ such that $\phi_{\mathbb C}^{-1}([J])$ generates $H^{2,0}(X,\mathbb C)$, where $\phi_\mathbb C$ is the extension of $\phi$ to $L\otimes \mathbb C$.  This also determines a complex structure on $X$, as we have the Hodge decomposition $H^2(X,\mathbb C)=H^{2,0}(X)\oplus H^{1,1}(X)\oplus H^{0,2}(X)$.  In other words, after fixing $\phi$, the period is given by a point in $\mathbb P(L\otimes C)$.  In particular, we define the period domain $\Omega$ to be
\[
\Omega=\{[J]\in \mathbb P(L\otimes C)\;\vert\;([J],[J])=0,\;\;([J],\overline{ [J]})>0\}.
\]
The global Torelli theorem, originally proven by \cite{PS}, states that every point in $\Omega$ corresponds to a marked K3 surface.  We define a period map $\tau_1:M_1\rightarrow \Omega$ from an analytic, non-Hausdorff, smooth space $M_1$ parametrizing marked K3 surfaces.  This map can be refined slightly as follows:  Together with the period point $[J]$, we can choose a K\"ahler class $\kappa\in H^{1,1}(X,\mathbb R)$.  In particular, $\kappa$ is characterised by the existence of a positive definite, with respect to $q$, plane $E([J])\subset H^2(X,\mathbb C)$, such that $q(\kappa,E([J]))=0$ and $q(\kappa,\kappa)>0$.  Define the set 
\[
K\Omega=\{(\kappa,[J])\in (L\otimes \mathbb R)\times \Omega\;\vert\;(\kappa,E([J]))=0,\;(\kappa,\kappa)>0\}
\]
where we define $E([J])$ to be the span of $\{Re[J],\;Im[J]\}$.  Then the set 
\[(K\Omega)^0=\{(\kappa,[J])\in K\Omega\;\vert\;(\kappa,d)\ne 0\;\mbox{ for any } d \mbox{ with } (d,d)=-2,\;([J],d)=0\}
\]
is open in $K\Omega$ and the refined period map $\tau_2:M_2\rightarrow (K\Omega)^0$ from a smooth Hausdorff analytic space $M_2$ parametrizing marked pairs $(X,\kappa)$ is an isomorphism.  This map is defined by $\tau_2(X,\kappa)=(\phi_\mathbb C(\kappa),\tau_1(X))$.  In particular, this isomorphism provides the following result:

\begin{lemma} \label{period} $e\in{\mathcal P}_X$ is a K\"ahler class if and only if
there is a q-positive definite $2-$plane $U$ in $H^2(X,{\mathbb R})$
such that $e\perp U$, and $q(e,d)\ne 0$ for any integral class $d$
in $H^2(X,{\mathbb Z})$ with $q(d,d)=-2$ and $d\perp U$.
\end{lemma}





The following Theorem is referred to as the Lefschetz (1,1) Theorem:

\begin{theorem}\label{lefschetz}(Thm IV.2.13, \cite{BPV})
\[
Pic(X)= H^{1,1}(X)\cap H^2(X,\mathbb Z)
\]
\end{theorem}

In particular, for any class $A\in Pic(X)$ with $A^2\ge -2$, either $A$ or $-A$ is effective.  Note that Pic(X) depends strongly on the complex structure, in fact, the rank of Pic(X), denoted $\rho(X,J)$,  can take all values in $[1,20]$.  Therefore, consider the following Lemma:

\begin{lemma}\label{pic}
Let $X$ be the $K3$-surface and $A\in H^2(X,\mathbb Z)$ with
$A^2\ge 0$.    Then there exist complex structures on $X$ such
that $A$ lies in the image of Pic(X) in $H^2(X,\mathbb C)$.
\end{lemma}

\begin{proof}By Theorem \ref{lefschetz}, we need to show that
$A\in H^{1,1}(X)$ for some decomposition given by a complex
structure $J$.  

If $A^2>0$, then $A$ is  a K\"ahler class and the complex structure determined by the plane $U$ from Lemma \ref{period} ensures $A\in Pic(X)$. 

Consider now the case $A^2=0$.  We shall make use of a nice feature for K3 surfaces, namely the
existence of a hyperk\"ahler metric $g$. Let $X$ be a marked K3 surface  with complex structure $J$ determined by the marking and fix a K\"ahler class $\omega$.  Then there exists a unique hyperk\"ahler metric $g$ of class $\omega$.  Moreover, this metric induces a
family of complex structures parameterized by the unit sphere of the
imaginary quaternions and a corresponding family of K\"ahler
forms. Denote this sphere by $S^2(g)$, and for each $u\in S^2(g)$,
denote the corresponding K\"ahler form by $\omega_u$. The span $F$
of the $\omega_u$ is a $3-$dimensional positive-definite subspace of
$H^2(X,\mathbb R)$ (with a basis given by $\{\omega_I, \omega_J,
\omega_K\}$). Let $F^{\perp}$ be the orthogonal complement of $F$,
then $H^2(X,{\mathbb R})=F\oplus F^{\perp}$ as $b^+=3$. In fact, $H^{1,1}_{u}={\mathbb R}[\omega_u]\oplus
F^{\perp}$.  Moreover, $J\in S^2(g)$.

If $A^2=0$, then $A\cdot e>0$ for some K\"ahler class $e$.  Define the Grassmannian $Gr^+_{A,e}$ of positive definite
2-planes which are orthogonal to $A$ and $e$.
This Grassmannian is nonempty by Prop 3.1, \cite{BL}.  Any element therein defines a complex structure such that $A\in H^{1,1}(X,\mathbb R)$.
\end{proof}

In the following, assume that $V$ is a submanifold in a K3 surface $X$ such that there exists a complex structure $J\in\mathcal J_V$ making $(X,J)$ a complex surface and $V$ a divisor.  In particular, $\mathfrak V$ lies in the N\'eron-Severi lattice of $(X,J)$.   We will call such a curve $V$ an algebraic curve.  Lemma \ref{pic} shows that such submanifolds exist for any class with non negative square.  In particular this implies that the Picard number $\rho(X,J)\ge 1$; K3-surfaces with this property are called algebraic K3-surfaces.  

Assume $V$ is an algebraic curve.  Then there exists a point $J\in\mathcal J_V$ which corresponds to a marking of $X$.  More precisely, let $\mathcal J_V^c\subset\mathcal J_V$ denote the subset of integrable almost complex structures.  Define the set
\[
\Omega_{\mathfrak V}=\{[J]\in\Omega\;\vert\;[J]\cdot\mathfrak V=0\},
\] 
this describes those markings of $X$ which make a submanifold in the class $\mathfrak V$ algebraic.  Then our assumption implies $\mathcal J_V^c\cap \Omega_\mathfrak V\ne \emptyset$.   Moreover, the following theorem shows that there exist markings such that the $\mathbb Z$-module $\langle\mathfrak V\rangle\subset L$ is the N\'eron-Severi lattice of $X$ and that such points are dense in $\Omega_{\mathfrak V}$:

\begin{theorem}(Cor II.12.5.3, \cite{S}; see also \cite{Mo1}, \cite{M})  Given  a sublattice $H$ of $L$ of rank $r$ such that the bilinear form restricted to $H$ has signature $(1,r-1)$, $r\le 20$, there exists an irreducible variety of dimension $(20-r)$ parametrizing a family of K3-surfaces $\{X_t\}$ with markings such that $H$ is a subset of the N\'eron-Severi group of any $X_t$.  Moreover, for generic $t$, $H$ is the N\'eron-Severi group of $X_t$.

\end{theorem}

This allows us to determine the relative Ruan invariant relative to an algebraic hypersurface $V$:

\begin{theorem}\label{k3vanish}Let $X$ be the K3 surface and assume there exists a marking of $X$ such that the class $\mathfrak V$ can be represented by an algebraic curve of genus 1 or higher.  Then for generic algebraic representatives $V$ of $\mathfrak V$, the invariant $Ru^V(A,[{\mathcal I_A}])$ vanishes for all $A\ne n\mathfrak V$.  If $\mathfrak V$ is not toroidal, then the invariant vanishes if $n>1$.

\end{theorem}

\begin{proof}

We have seen, that there exists a dense $U\subset \Omega_\mathfrak V$ such that for any $[J]\in U$, we have $\rho(X,J)=1$.  This means that the only holomorphic curves are in class $\mathfrak V$.  Hence for any embedded algebraic curve $V$ representing $\mathfrak V$, we can be sure there are no relative embedded $J$-holomorphic submanifolds in any class other than possibly $A=n\mathfrak V$.

Now apply the vanishing principle in \cite{LP} to show that all the relative invariants $Ru^V(A,[{\mathcal I_A}])$ vanish.

\end{proof}

\begin{lemma}
The same holds true for any symplectic hypersurface $V$ in class $\mathfrak V$ such that the deformation class of $(X,V,\omega)$ contains an algebraic representative of $\mathfrak V$.
\end{lemma}

{\bf Remark:}  The restriction on the genus in Theorem \ref{k3vanish} ensures that $d_\mathfrak V=\mathfrak V^2\ge 0$ is fulfilled, hence $V$ is stable.  The missing genus 0 case is the only non-stable case.

The definition of the relative Ruan invariant is not especially helpful in this setting towards an actual computation of the invariant.  Such calculations will be made considerably easier through the use of a rubber-type calculus,  which will be developed in subsequent papers.  The results obtained by Maulik and Pandharipande in \cite{MP2} make use of localization techniques.

\subsection{Refined relative Ruan Invariants\label{refined}}

The purpose of this Section is to describe how two curves $C$ and $C'$ in $\mathcal K_V(A,J,\mathcal I_A)$ can be distinguished even though both curves lie in the  same class $A$ and meet the same set of initial data $\mathcal I_A$.   This will be used in the next Section to define a  refinement of the invariant $Ru^V(A,[{\mathcal I_A}])$.  We will describe here the construction used in \cite{IP4}.

\subsubsection{\label{rimtori}Rim Tori}
The difference $C\#\overline{C'}$ of the curves $C$ and $C'$ lies not in $X$ but in the open manifold $X/ V$.  More precisely, the class of the difference lies in the kernel $\mathfrak R$ of the map $H_2(X\backslash V)\rightarrow H_2(X)$.  The key is to find an optimal space with which to describe this difference while keeping track of the data $A$ and $\mathcal I_A$.  To that end, let us fix notation: For a given class of initial data $[\mathcal I_A]$, let $V_{[\mathcal I_A]}$ denote the collection of all sets of pairs $((x_1,s_1),...,(x_l,s_l))$ of intersection points in $V$ and contact orders to be found in initial data in class $[\mathcal I_A]$.  Define 
\[
V_A=\bigsqcup_{[\mathcal I_A]} V_{[\mathcal I_A]}
\]
 with the topology of the disjoint union.  Note that this space has an induced ordering on each point $V_{[\mathcal I_A]}$ coming from the ordering on the class ${[\mathcal I_A]}$.  Let $D(\epsilon)$ be an $\epsilon$-disk bundle in the normal bundle to $V$.  Then $X/\overline{D(\epsilon)}$ is diffeomorphic to $X/V$.  Define the space 
\[
\hat X=[X/\overline{D(\epsilon)}]\cup S
\]
where $S=\partial D(\epsilon)$.  The manifold $\hat X$ is compact and, endowing $S$ with the topology given by viewing it as a disjoint union of its fiber circles, we can consider the long exact sequence of the pair $(\hat X,S)$:
\[
0\rightarrow H_2(\hat X)\rightarrow H_2(\hat X,S)\rightarrow H_1(S)\rightarrow.
\]
In the given topology for $S$, the set $H_1(S)$ can be viewed as the space of divisors on $V$, meaning the finite collection of points labeled with multiplicities and sign.  This however is precisely the data in $\mathcal I_A$ relating to the intersection of the curves $C$ and $C'$ with the hypersurface $V$ (the sign is always $+$).  Note however, that $H_1(S)$ does not come with an ordering, it makes no distinction between data for curves meeting $V$ in the same points with the same multiplicity but with a differing ordering on the contact data.

Combining this sequence with the map $\pi:H_2(\hat X,S)\rightarrow H_2(X)$ induced by the inclusion, which has as its kernel the set $\mathfrak R$, leads to the exact sequence
\begin{equation}\label{rimseq}
0\rightarrow \mathfrak R\rightarrow H_2(\hat X,S)\rightarrow H_1(S)\times H_2(X).
\end{equation}
There is a map from $V_A$ to the set of divisors $H_1(S)$ which maps onto the set of effective divisors.  This allows for the definition of the space $\mathcal H_V^X$ by the following pullback diagram:

\begin{equation}\label{rimspace}
\begin{diagram}
\node{\mathcal H_V^X}\arrow{e}\arrow{s}\node{V_A}\arrow{s}\\
\node{H_2(\hat X,S)}\arrow{e}\node{H_1(S)}
\end{diagram}
\end{equation}
Combining \ref{rimseq} and \ref{rimspace} we obtain the fibration
\begin{equation}\label{rimlift}
\begin{diagram}
\node{\mathfrak R}\arrow{e}\node{\mathcal H_V^X}\arrow{s,r}{\mathfrak r}\\
\node[2]{H_2(X)\times V_A}
\end{diagram}
\end{equation}
which allows us to lift a class $A$ and its initial intersection data to a point in $\mathcal H_V^X$ which encodes the information on the intersection data as well as the class of the curve $C$ in the kernel $\mathfrak R$.  The procedure for a given curve $C$ is as follows:  Restrict the curve $C$ to $X\backslash V$, lift to the space $\hat X$ and use the construction in \cite{LR} (see also the brief discussion in Section \ref{contactorder}) to close the restricted curve $C$ to a curve $\hat C\subset \hat X$.  The class $[\hat C]\in H_2(\hat X,S)$ together with the intersection data from $C$ defines a point in $\mathcal H_V^X$.

In order for this construction to be useful, we need a characterisation of the kernel $\mathfrak R$.  This has been given in \cite{IP4}:  Let $\pi:S\rightarrow V$ be the projection map from the boundary of the $\epsilon$-disk bundle to the hypersurface $V$.  For every simple closed loop $\gamma$ in $V$, $\pi^{-1}(\gamma)$ is a torus in $S$.  Such tori are called rim tori and they generate $\mathfrak R$:

\begin{lemma}
(Lemma 5.2, \cite{IP4}) Each element in $\mathfrak R$ can be represented by a rim torus.
\end{lemma}

The proof of this Lemma utilises the Gysin sequence for the oriented circle bundle $\pi:S\rightarrow V$ and the Meyer-Vietoris sequence of $(X,X\backslash V,V)$:
\[
\rightarrow H_3(V)\rightarrow H_1(V)\stackrel{\Delta}{\rightarrow}H_2(S)\rightarrow H_2(V)\rightarrow
\]
\[
\rightarrow H_2(S)\stackrel{(\iota_*,\pi_*)}{\rightarrow} H_2(X\backslash V)\oplus H_2(V)\rightarrow H_2(X)\rightarrow
\]
This leads to
\[
\mathfrak R=\mbox{image }[i_*\circ \Delta:H_1(V)\rightarrow H_2(X\backslash V)]
\]
which eliminates those rim tori which are homologous to zero in $X\backslash V$.   We will call the set $\mathfrak R$ the set of rim tori in the following.

\subsubsection{Refined relative Ruan Invariants}

Rim tori allow us to differentiate curves lying in the same set $\mathcal K_V(A,J,\mathcal I_A)$ and thereby will allow us to refine the invariant $Ru^V(A,[\mathcal I_A])$.  We first describe how this refinement works and then concern ourselves with the properties of this definition.  Note that this refinement is not interesting for all classes $A\in H_2(X)$: If $A\cdot \mathfrak V\le 0$, then our results show that either $\mathcal K_V(A,J,\mathcal I_A)=\emptyset$ or $A\cdot \mathfrak V=0$ .

To each point $h\in \mathcal K_V(A,J,\mathcal I_A)$ can be associated a class $\hat A\in \mathcal H_X^V$.


It is therefore possible to decompose the spaces  $\mathcal K_V(A, J, \mathcal I_A)$ such that
\[
\mathcal K_V(A,J,\mathcal I_A)=\bigsqcup_{\hat A}\mathcal K_V(A, \hat A, J, \mathcal I_A)
\]
 according to possible lifts of the point $(A,\mathcal I_A)$ under the map \ref{rimlift}.  This also serves as a definition of the spaces $\mathcal K_V(A,\hat A,J,\mathcal I_A)$.  It is immediately clear that these spaces are finite and smooth for generic pairs $(J,\mathcal I_A)$.   Moreover, the following Proposition, analogous to Prop. \ref{mainprop}, holds:

\begin{prop}Fix a class $A\in H_2(X)$ with $A\cdot \mathfrak V>0$ and a proper class $[\mathcal I_A]$.  Then there is a Baire subset of $\mathcal J_V\times [\mathcal I_A]$ such that

\begin{enumerate}
\item The set ${\mathcal K}_V(A,\hat A,J,{\mathcal I_A})$ is a finite set.
\item Every point $h\in {\mathcal K}_V(A,\hat A,J,{\mathcal I_A})$ has the property, that each $C$ is non-degenerate.
\item If $(J^1,\mathcal I_A^1)$ is close to $(J,\mathcal I_A)$, then the sets ${\mathcal K}_V(A,\hat A,J,{\mathcal I_A})$ and ${\mathcal K}_V(A,\hat A,J^1,{\mathcal I_A}^1)$ have the same number of elements.
\end{enumerate}
\end{prop}

\begin{proof}
Only the last claim is not obvious from the results for $\mathcal K_V(A,J,\mathcal I_A)$.  The issue is whether a small change in the data might cause rim tori to disappear or to be generated.  However, this can be ruled out due to the fibration structure in \ref{rimlift}. 
\end{proof}

The definition of the refined invariant now follows from the definitions in Section \ref{relinv}:

\begin{definition}\label{rimrelgr} The refined Ruan invariant for the symplectic hypersurface $V\subset X$ and the class $A\in H_2(X)$ and with initial data class $[\mathcal I_A]$ is denoted $Ru^V(\hat A,[{\mathcal I_A}])$ and is defined by
\begin{equation}
Ru^V(\hat A,[\mathcal I_A])=\sum_{C\in \mathcal K(A,\hat A,J,\mathcal I_A)}r(C,\mathcal I_A).
\end{equation}
\end{definition}

The following is trivial:
\begin{lemma}
\begin{equation}
Ru^V(A,[{\mathcal I_A}])=\sum_{\mathfrak r^{-1} A}Ru^V(\hat A,[{\mathcal I_A}])
\end{equation}
\end{lemma}

The invariant properties of $Ru^V(\hat A,[{\mathcal I_A}])$ are more subtle than in the non-refined case.  Using cobordism arguments as in Section \ref{relinv}, it can be shown that $Ru^V(\hat A,[{\mathcal I_A}])$ is an invariant of the symplectic isotopy class $[X,V,\omega]$.  Furthermore, deformations of only the symplectic structure $\omega$ also leave this number invariant.  However, the numbers $r(C,\mathcal I)$ depend on the orientation of the normal bundle of $V$ as well as the almost complex structure on the normal bundle of $V$, hence it is unlikely that it is invariant under deformations of $V$.  This question remains open.  The families of manifolds constructed in \cite{FS}, \cite{Sm} and \cite{HP} should provide a plethora of examples for which to calculate these invariants.

\section{\label{taubesinv}Relative Taubes Invariant for Tori with Trivial Normal Bundle}

Taubes defined an invariant counting tori with trivial normal bundle in \cite{T4}.  This invariant takes into account the bifurcation behavior of sequences of such tori.  In this section, we show that this delicate count can be done in the relative settings without any modification of the Taubes invariant.
\subsection{Behavior of Multiply Toroidal Classes}

\subsubsection{\label{multori} Non-Degeneracy of Multiply Toroidal Classes}
Special consideration must be given to classes representing square 0 tori.  This issue will occur throughout the following sections.  For this reason, we make the following definition: 

\begin{definition}
A class $A\in H_2(X)$ is called multiply toroidal if
\begin{itemize}
 \item $A^2=0$,
\item $ K_\omega\cdot A=0$ and
\item  the class $A$ is divisible, i.e. $A=kA'$ with $k>1$.  
\end{itemize}
We call the class $A$ toroidal if the first two conditions hold.
\end{definition}

If $C$ is a torus with trivial normal bundle, then we expand the definition of non-degeneracy:

\begin{definition}\label{0tori}Fix an almost complex structure $J\in\mathcal J_V$.  If $C$ is a torus with trivial normal bundle, fix a positive integer $n\in \mathbb Z$ and call $C$ n-non-degenerate if $C'$ is non-degenerate for every holomorphic covering map $f:C'\rightarrow C$ of degree n or less.
\end{definition}

We can make this more precise:  For any representatives of a (multiply) toroidal class, it is not possible to distinguish different values of $k$ by marked points.  Therefore, all such curves must be considered when constructing the invariant.  Multiple covers of a torus with trivial normal bundle are classified by the fundamental group $\pi_1(C)=\mathbb Z\oplus \mathbb Z$.  Consider a homomorphism $\rho:\pi_1(C)\rightarrow P_m$, $P_m$ the permutation group on $m$ letters.  $\rho$ defines, via a representation of $P_m$ on $\mathbb R^m$, a m-plane bundle $V_\rho$ over $C$.  This will allow us to distinguish multiple covers of the base curve $C$ in the normal bundle.  We can naturally extend the operator $D$ to the space of sections of $V_\rho\times N$.  We can now make precise the definition of $n$-nondegenerate:

\begin{definition}Let $C$ represent a (multiply) toroidal class.  Fix $n\in\mathbb Z$.  The curve $C$ is $n$-nondegenerate if for all $m\in\{1,..,n\}$ and for all representations $\rho:\pi_1(C)\rightarrow P_m$ the operator $D$ on the space $V_\rho\times N$ has trivial kernel.

\end{definition}

This definition ensures, that any curves which are counted and which stem from multiple covers of the curve $C$ behave well.   In particular, the following Lemma was proven by Taubes:

\begin{lemma}(Lemma 5.4, \cite{T4}) Let $A\in H_2(X)$ be toroidal and $n\in\mathbb Z^+$.  Then there is an open and dense subset of smooth, $\omega$-compatible almost complex structures $J\subset \mathcal J_\omega$ on $X$ with the property, that every embedded, pseudoholomorphic torus in class $A$ is $n$-nondegenerate.

\end{lemma}

Further arguments, similar to those in step 7 of the proof of Prop. 7.1 in \cite{T1}, together with this Lemma ensure that the space of curves in class $A$ is finite and invariant under small deformations of the symplectic and almost complex structures.

\subsubsection{Contact Order for Tori with Trivial Normal Bundle}

If the curve $C$ is a multiply covered torus with trivial normal bundle which intersects $V$ non-trivially, then we must account for this in the contact order.  Lemma \ref{values} states, that in this case, to have proper initial data,  we must have the base curve $C'$ of $C$ intersecting $V$ with contact order 1 at each intersection point, where the contact order of $C'$ and $V$ is given by Def. \ref{contord}.  Thus we make the following definition

\begin{definition}\label{toricontact}Assume $C$ is a multiply covered torus with trivial normal bundle in class $A=mA'$ such that $A'\cdot \mathfrak V\ne 0$.  Then the contact order of the curve $C$ at each intersection point of $C'$ with $V$ is given by an $m$-tuple $(1,...,1)$.

\end{definition}

Taubes' calculations all have initial data  $\mathcal I_A=\emptyset$.  The relative setting must deal with the presence of initial data on $V$:

\begin{lemma}Let $A=mT$ be multiply toroidal.  Curves in this class can occur only if the only entries in $\mathcal I_A$ are in $\Upsilon$ each with an $m$-tuple $(1,...,1)$ as contact order.

\end{lemma}

\begin{proof}A multiply toroidal class has $d_T=0$, hence by Lemma \ref{values}, we have $d_1=d_2=l_1=l_2=0$ and $l_3=T\cdot \mathfrak V$.  Moreover, Def. \ref{toricontact} determines the contact structure.
\end{proof}

\subsection{Space of Relative Tori}  We define a space of relative submanifolds: 

\begin{definition}\label{defKT}Fix $A\in H_2(X)$ and a set of proper initial data $\mathcal I_A$.  Assume that $A=mT$ is multiply toroidal.  Choose an almost complex structure $J\in\mathcal J_V$.  Denote the set $\mathcal K^T_V= {\mathcal K}^T_V(A,J,\mathcal I_A)$ of connected $J$-holomorphic submanifolds $C\subset X$ which satisfy

\begin{itemize}
 
\item If $l_{A}=0$, then $C\cap V=\emptyset$ or $V$.
\item If $l_{A}>0$, then $C$
\begin{enumerate}
\item intersects $V$ locally positively and transversely and 
\item intersects $V$ at precisely $l_3$ distinct points each with order given by the $m$-tuple $(1,...,1)$.
\end{enumerate}
\end{itemize}
\end{definition}

We wish to show that Taubes' Lemmas still hold in the relative case.   We consider three cases:
\begin{itemize}
\item $A_i\cdot \mathfrak V\ne 0$.
\item $A_i\ne \mathfrak V$ and $A_i\cdot \mathfrak V=0$ or
\item $A_i=\mathfrak V$ and $V$ is a square 0 torus.  
\end{itemize}
The previous Lemma shows that we do not have to consider any restricting insertions.  We begin with the first case.

\begin{lemma}\label{A=V}
Assume $A$ is multiply toroidal and fix $n\in \mathbb N$.  There is an open and dense subset $U\subset \mathcal J_V$ with the following properties:  When an almost complex structure $J$ is chosen from $U$, then

\begin{enumerate}

\item $\mathcal K^T_V(J,\mathcal I_A)$ is a finite collection of points and each point is $n$ - nondegenerate. Moreover, if  $A\ne m\mathfrak V$ and $A\cdot \mathfrak V=0$, there exists a neighborhood $\mathfrak N_V$ of $V$ such that no curve in class $A$ lies therein. 
\item There is an open neighborhood in $\mathcal J_V$ such that every almost complex structure therein obeys the previous assertion and the number of points of $\mathcal K_V^T$ is invariant in this neighborhood.
\end{enumerate}

\end{lemma}

\begin{proof} {$\bf A\cdot \mathfrak V\ne 0$}  The calculations in Section \ref{generic} ensure the existence of a sufficiently large set of almost complex structures such that curves in class $A$ behave as expected.  Once this set has been found, the proofs of Lemma 5.3 and 5.4 in \cite{T4} can be repeated to give the result in this case.

{$\bf A\cdot \mathfrak V=0$}  Assume first that $A\ne m\mathfrak V$.  We will show, that there exists a neighborhood $\mathfrak N_V\subset X$ of $V$ such that no pseudoholomorphic submanifold $C$ of class $A$ has $C\cap \mathfrak N_V\ne \emptyset$.  If this holds, then all results on multiply toroidal classes from \cite{T} hold for the class $A$:  Outside of $\mathfrak N_V$ there are no restrictions on the almost complex structure $J\in\mathcal J_V$, hence all the results proven by Taubes for multiply toroidal classes (Lemmas 5.3,5.4) hold in this case as well.  This proves the Lemma.

Assume no such neighborhood exists.  Then let $\{\mathfrak N_i\}$ be a sequence of nested neighborhoods converging to $V$.  Denote $A=n\mathfrak T$ with $\mathfrak T$ a toroidal class.  Let $\{T_i,J_i\}$ be a sequence of tori in class $\mathfrak T$ such that $T_i\subset \mathfrak N_i$ and $J_i\in\mathcal J_V$.  Gromov's compactness Theorem ensures that there exists a limit curve $T$ and a limit almost complex structure $J_i\rightarrow J$ such that $g(T)=1$ and Lemma \ref{conv} shows that for generic $J$ the limit curve $T$ is an embedded square 0-torus of class $q\mathfrak T$.  The sequence of neighborhoods ensures that $T\subset V$.  If $g(V)\ge 2$, then there can be no nontrivial smooth map $T^2\rightarrow V$.

Assume that $g(V)=1$.  Then we have produced a $J$-holomorphic map $T^2\rightarrow V$ in class $A\ne m\mathfrak V$ which is onto $V$.  Such a nontrivial map does not exist.

Assume that $g(V)=0$.  Then we have a $J$-holomorphic map from a torus $T^2$ to the sphere.  This must be a multiple cover of $S^2$.  However, $K_\omega\cdot A=K_\omega\cdot\mathfrak T=0$ and hence by the adjunction formula there exists no such map.

Consider now the second case: $A=m\mathfrak V$, hence any curve in this class can be decomposed into components such that we have either a multiple cover of $V$ or a multiple cover of a curve in the class $\mathfrak V$ which does not meet $V$.  For those curves not meeting $V$, arguments similar to the previous ones show that Taubes' results hold.  We consider therefore only the case of a multiple cover of the hypersurface $V$.  An analysis of Taubes' results shows, that the argument in the multiply toroidal case is essentially a relative argument on the fixed square 0 torus underlying the multiple cover.  In our case this is the hypersurface $V$, hence his argument transfers completely.  

\end{proof}

\subsection{Relative Taubes Invariant}

Taubes defined a number $Qu(e,n)$ for multiply toroidal classes and showed that it is an invariant of the deformation class of $\omega$(Prop. 5.7, \cite{T4}, see also Prop \ref{taubesprop}).  This motivates the following slightly modified relative version:

\begin{definition}
Let $n\ge 1$ be an integer and $T$ toroidal and indivisible.  Choose $(J,\mathcal I_A)$ from the Baire set obtained in Lemma \ref{A=V}.  Define a relative Taubes invariant
\[
Qu^V(T,n)=\sum_{\{(C_k,m_k,\mathcal I_k)\}}\prod_k r(C_k,m_k,\mathcal I_k)
\] 
where we sum over all sets $\{(C_k,m_k,\mathcal I_k)\}$ with
\begin{enumerate}
\item $C_k$ an embedded torus in class $q_kT$ for some $q_k\le n$,
\item $m_k\ge 1$ and $n=\sum q_km_k$ and
\item $\mathcal I_k$ contains tuples $(1,...,1)$ of length $q_k$ at each of the $T\cdot\mathfrak V$ intersection points.
\end{enumerate}
\end{definition}
The value of $r(C,m,\mathcal I_k)$ is given by the value of $r(C,m)$ as defined by Taubes for multiply toroidal classes in \cite{T4}.

The results in Lemma \ref{A=V} show that only if $A\cdot \mathfrak V\ne 0$ do we need to reconsider the definition of the invariant $Qu(T,n)$ as given by Taubes.  However, an analysis of the proof of Proposition 5.7, which states that $Qu(e,n)$ is an invariant of the deformation class of $\omega$ shows that this proof too relies on the existence of a sufficiently generic set of almost complex structures such that curves behave as expected for any $J$ chosen therein.  Hence the invariant defined by Taubes can be directly used in the relative setting with the modifications given above, i.e. $Qu^V(T,n)$ is also an invariant of the deformation class of $(X,V,\omega)$.

\section{Relative Gromov-Taubes Invariants\label{rgt}}

The invariants defined in the previous sections are concerned with counting of connected submanifolds.  In this section we expand this to allow disconnected invariants as in the Gromov-Taubes invariants defined in \cite{T4}.  These relative GT-invariants will make use of the invariants of the previous sections.

\subsection{The Space of Relative Submanifolds}

We now introduce the space of relative submanifolds ${\mathcal R}_V(A,J,{\mathcal I_A})$.  This definition will be rather technical, however the general idea is simple:  We want to consider all submanifolds $C$, not necessarily connected, which contact $V$ in a very controlled manner.  This is determined by the initial data $\mathcal I_A$ and we ensure that we contact $V$ only once for every given geometric object with the required contact order.  Moreover, the curve $C$ shall meet each geometric object in the initial data $\mathcal I_A$.  We make this precise in the following definition:

\begin{definition}\label{defR}Fix $A\in H_2(X)$ and a set of proper initial data $\mathcal I_A$.  Choose an almost complex structure $J\in\mathcal J_V$.  Denote the set ${\mathcal R} = {\mathcal R}_V(A,J,\mathcal{I}_A)$ of unordered sets of tuples $\{(C_i,m_i,{\mathcal I}_i)\}$ of disjoint, connected $J$-holomorphic submanifolds $C_i\subset X$ with

\begin{enumerate}
\item positive integers $m_i$ and
\item unordered subsets ${\mathcal I}_i$ of ${\mathcal I_A}$ with parameters $d_*^i, l_*^i$ 
\end{enumerate}

satisfying the following constraints:

\begin{itemize}
\item Let $A_i$ represent the homology class of $C_i$ and denote $d_{A_i}$ as in \ref{d} and $l_{A_i}$ as in \ref{l}.  Require $d_{A_i}\ge 0$, $l_{A_i}\ge 0$ and the set $\mathcal I_i$ to be a proper initial data set for the class $A_i$. 
\item If $d_{A_i}>0$, then $C_i$ 
\begin{enumerate}
\item contains precisely $d_1^i$ members of $\Omega_{d_1}$ and
\item  intersects each member of $\Gamma^X_i$ exactly once.
\end{enumerate}
\item If $l_{A_i}=0$, then $C_i\cap V=\emptyset$ or $V$.
\item If $l_{A_i}>0$, then $C_i$
\begin{enumerate}
\item intersects $V$ locally positively and transversely,
\item intersects $V$ at precisely $l_1^i$ points of $\Omega_{l_1}$ and
\item  intersects each member of $\Gamma^V_i$ exactly once.  
\item The remaining $l_3^i$ intersections with $V$ are unconstrained.
\item Each intersection is of order $s_i$ given in the initial data $\mathcal I_i$ for this component (or by the tuple $(1,..,1)$ if it is multiply toroidal). 
\end{enumerate}
\item If $i\ne i'$, then ${\mathcal I}_i\cap {\mathcal I}_{i'}=\emptyset$, but $\cup {\mathcal I}_{i}={\mathcal I}_A$.

\item The integer $m_i=1$ unless possibly if $C_i$ is a torus with trivial normal bundle and $A_i\cdot V=0$. 
\item $\sum_i m_i A_i=A$.
\end{itemize}
\end{definition}

{\bf Remark:} 
By imposing the condition that the $C_i$ be disjoint allows us to conclude
\begin{equation}
\sum_i d_1^i+d_2^i=d_A,\;\;\;\sum_i l_1^i+l_2^i+l_3^i=l_A.
\end{equation}

The points of $\mathcal R$ consist of submanifolds which are not necessarily connected.  Therefore, a choice of almost complex structure $J$ and initial data $\mathcal I_A$ must be made in such a manner, that any allowed decomposition of the class $A$ into submanifolds respects the initial data and that all these submanifolds are pseudoholomorphic for the fixed almost complex structure.  In particular, the pair $(J,\mathcal I_A)$ must be chosen such that it rules out any unwanted behavior of representatives of the class $A$ and its decompositions.  This motivates the following definition:

\begin{definition}\label{rad}A pair $(J,{\mathcal I})$ of almost complex structure $J\in {\mathcal J}_V$ and initial data ${\mathcal I}$ with $d_1+d_2=d$ and $l_1+l_2+l_3=l$ is called r-admissible, if for each $A\in H_2(X)$ and each proper initial data  $\mathcal I_A\subset \mathcal I$ the following conditions hold:

\begin{enumerate}
\item There are but finitely many connected $J$-holomorphic submanifolds in the class $A$ contacting the initial data $\mathcal I_A$.
\item Each of the submanifolds above is non-degenerate.
\item There exist no connected $J$-holomorphic submanifolds in class $A$ contacting all the data in $\mathcal I_A$ as well as a further insertion.
\item There is an open neighborhood of $(J,\mathcal I)$ in $\mathcal J_V\times [\mathcal I]$ with the property that each point in this neighborhood obeys the previous three points while preserving the number of $J$-holomorphic curves throughout this neighborhood.
\item If $A^2=0=c_1\cdot A$, then each $J$-holomorphic submanifold in Point 1 is n-non-degenerate for each positive integer $n$
\end{enumerate}
\end{definition}

\subsection{Main Result for Disconnected Submanifolds}

\begin{prop} Fix a class $A\in H_2(X)$ and a proper class $[\mathcal I_A]$.  Assume $V$ is a symplectic hypersurface.  Then the set of r-admissible pairs $(J,\mathcal I_A)$ in $\mathcal J_V\times [\mathcal I_A]$ is a Baire subset.  Furthermore, given an r-admissible pair, the following hold:

\begin{enumerate}
\item The set ${\mathcal R}_V(A,J,{\mathcal I_A})$ is a finite set.
\item If $A\ne \mathfrak V$, then ${\mathcal R}_V(A,J,{\mathcal I_A})$ is empty when $d_A<0$. 
\item If $V$ is an exceptional sphere, then ${\mathcal R}_V(\mathfrak V,J,\emptyset)=\emptyset$
\item Every point $h\in {\mathcal R}_V$ has the property, that each $C_i$ with $m_i=1$ is non-degenerate, if $m_i>1$ it is $m_i$-non-degenerate.
\item If $(J^1,\mathcal I_A^1)$ are sufficiently close to $(J,\mathcal I_A)$, then the sets ${\mathcal R}_V$ and ${\mathcal R}_V^1$ have the same number of elements.
\end{enumerate}
\end{prop}

Fix a symplectic hypersurface $V$.  Lemma \ref{index} ensures the existence of a Baire set with pairs $(J,\mathcal I_A)$ such that no component of a submanifold representing $A$ lies in $V$.  We may hence assume that all of our submanifolds are of this form in the following. 

The set of r-admissible pairs $(J,\mathcal I_A)$ is Baire follows from the fact that a countable intersection of Baire sets is again Baire.  To be precise:  Consider a decomposition of $A=\sum m_kA_k$ with $A_k\cdot A_l=0$ if $k\ne l$.  Then for each $A_k$ we consider the space of connected submanifolds $\mathcal K$.  For a Baire set of pairs $(J,\mathcal I_{A_k})$ the set $\mathcal K_V$ has all the properties described in the previous sections.  The space $\mathcal J_V\times \{\mbox{initial data} \}$ decomposes into a product of $\mathcal J_V$ and disjoint initial data sets corresponding to the classes $A_k$.  Then the intersection of two Baire sets $U_k$ and $U_l$ corresponding to $A_k$ and $A_l$ is defined as follows:  We have endowed all of the above products with the product topology.  Denote $p_1$, $p^*_2$ the projections onto $\mathcal J_V$ and the initial data set corresponding to $A_*$.  Then let
\begin{equation}
U_k\cap U_l= p_1(U_k)\cap p_1(U_l)\times p_2^k(U_k)\times p_2^l(U_l)
\end{equation}
define the intersection of the two Baire sets.  This set is still a Baire set in $\mathcal J_V \times \{\mbox{initial data} \}_k \times \{\mbox{initial data} \}_l$ due to the properties of the product topology.  From this the claim follows.

The properties of being r-admissible as given in Def. \ref{rad} are fulfilled by pairs obtained by a countable intersection of all Baire sets found in Sections \ref{generic} and \ref{structure} for any decomposition of $A$ appearing in $\mathcal R_V$.  Moreover, any pair $(J,\mathcal I_A)$ found in this intersection also satisfy assertions 2-4 of Prop. \ref{mainprop}.

Consider now Assertion 1: $\mathcal R_V(A,J,\mathcal I_A)$ is finite.  This will follow from the following result which was proven by Taubes:

\begin{lemma}\label{finiteA}(Lemma 5.5, \cite{T4})  Given $A\in H_2(X)$, there is a Baire subset of ${\mathcal J}_V\times [\mathcal I_A]$ such that when a pair is chosen from this set, then there are but finitely many classes in $H_2(X,\mathbb Z)$ which can be a fundamental class of a $J$-holomorphic submanifold appearing (with some multiplicity) as an element in some $h\in{\mathcal R}_V$.
\end{lemma}

\subsection{The Number $GT^V(A)([{\mathcal I_A}])$}

To define  $GT^V(A)([{\mathcal I_A}])$ we will need to begin with a component of $h=\{(C_k,m_k,\mathcal I_k)\}\in \mathcal R_V(A,J,\mathcal I_A)$.  To each such component, we will assign a number $r(C_k,m_k,\mathcal I_k)$.  Once this has been defined, we will define a value $q(h)$ for the point $h$, this will consist of the values $r(C_k,m_k,\mathcal I_k)$ as well as accounting for permutations of the initial data.  Finally, we define the relative invariant:
\begin{definition}\label{relgr} The relative Gromov-Taubes invariant $GT^V(A)([{\mathcal I_A}])$ is defined by
\begin{equation}
GT^V(A)([{\mathcal I_A}])=\sum_{h\in \mathcal R_V(A,J,\mathcal I_A)}q(h).
\end{equation}
If $\mathcal R_V(A,J,\mathcal I_A)=\emptyset$, then $GT^V(A)([{\mathcal I_A}])=0$.
\end{definition}
For a suitably generic choice of pairs $(J,\mathcal I_A)$, this number is well-defined.  In the following, we make this definition precise and show that this number is an invariant of the symplectic deformation class.

{\bf Remark:}  This number depends not only on the class $[\mathcal I_A]$, but actually also on the ordering given in the sets $\Gamma_*$.  Ultimately, this only affects the sign of $GT^V(A)([{\mathcal I_A}])$.  This will be taken into account in the definition of $q(h)$.

\subsection{The Definition of $q(h)$}
\subsubsection{Permutations of the Initial Data}
A point $h\in\mathcal R$ need not meet the data  $\mathcal I_A$ in the order predicated in the class $[\mathcal I_A]$.  This ordering determines an orientation of the corresponding moduli space of maps however and thus any permutation of it must be taken into account when defining an invariant.  Our data $\mathcal I_A$ contains three types of geometric data: points, 1-dimensional curves and the hypersurface $V$.  Neither the points nor the hypersurface $V$ change the orientation under rearrangement.  Recall the sets $\Gamma_{d_2}$ and $\Gamma_{l_2}$ and consider only the curves and the ordering.  We can rearrange each set by a permutation $\pi_{d_2}$ resp. $\pi_{l_2}$ as follows:  Consider the point $h=\{(C_k,m_k,\mathcal I_k)\}$.  This point comes with an ordering.  For each $i$ define the sets $\Gamma_1^k\subset \Gamma_{d_2}$ and $\Gamma_2^k\subset \Gamma_{l_2}$ consisting of the data in $\mathcal I_k$ in the corresponding sets.  Reorder the data in $\Gamma_*^k$ in ascending fashion according to the ordering given in $\mathcal I_A$.  Then $\sqcup \Gamma_*^k$ defines a permutation of the data in $\Gamma_{d_2}$ resp. $\Gamma_{l_2}$.

\begin{definition}\label{p}
Define $p(h)=\mbox{sign}(\pi_{d_2})\mbox{sign}(\pi_{l_2})$.
\end{definition}

Equivalently, we could consider the set $\Gamma=\Gamma_{d_2}\sqcup\Gamma_{l_2}$ and a corresponding permutation $\pi$ which consists of the two permutations $\pi_{d_2}$ and $\pi_{l_2}$.  Then $p(h)=$sign$(\pi)$.

The ordering on the curve $h$ is not fixed.  We must show that a relabeling of the curves $C_k$ will leave the value of $p(h)$ unchanged.  Let $d_2^k$ and $l_2^k $ denote the number of elements in $\Gamma_1^k$ resp. $\Gamma_2^k$.  Then the invariance of $p(h)$ under reordering follows from 

\begin{lemma}
The value $d^k_2+l_2^k$ is even.
\end{lemma}

\begin{proof}
We have shown, that for generic $(J,\mathcal I_A)$ we can have curves in the class $A$ only if  the initial data is proper.  This holds in particular for every connected component of the curve $h$.  The condition for properness can be easily rewritten to show that $d^k_2+l_2^k$ is even.

\end{proof}

\subsubsection{$q(h)$}
The value of $q(h)$ for the point $h=\{(C_k,m_k,\mathcal I_k)\}$  is given by the product
\begin{equation}\label{q(h)}
q(h)=p(h)\prod_k r(C_k,m_k,\mathcal I_k)
\end{equation}
where the $r(C,m,\mathcal I)$ are given by $r(C,\mathcal I)$ as defined for non-multiply toroidal classes in Section \ref{relinv} and by Taubes in the toroidal case.

\subsection{Properties of Relative Gromov-Taubes Invariants}

\begin{theorem}\label{invG}
The number $GT^V(A)([{\mathcal I_A}])$ depends only on the deformation class of $(X,V,\omega)$, the class $A\in H_2(X)$, the initial class $[\mathcal I_A]$ and the ordering of the data in the sets $\Gamma_*$.  In particular, it does not depend on a particular choice of $(J,\mathcal I_A)$.
\end{theorem}

To prove this, we proceed to rewrite the number $GT^V(A)([\mathcal I_A])$ in terms of the number $Ru^V(A,[\mathcal I_A])$ and the toroidal contributions $Qu^V(A,m)$.  To do so we introduce notation:  Let $A\in H_2(X)$ be fixed.  Denote by $S(A)$ the set defined by Taubes: This is the collection of unordered sets of pairs $\{(A_k,m_k)\}$ with the following properties:
\begin{enumerate}
\item $\{A_k\}$ is a set of distinct, non-multiply toroidal classes.
\item $m_k=1$ unless $A_k^2=0$,  in which case $m_k\ge 1$ can be any positive integer.
\item $A_k\cdot A_l=0$ if $k\ne l$.
 \item $A=\sum m_kA_k$.
\end{enumerate}

Note that it is possible for $m_k\ge 2$ but for $A_k$ not to be a toroidal class:  The set $S(A)$ is a set of homology classes, we allow our submanifolds to be composed of multiple disjoint copies of classes with 0 self intersection.  This is taken into account by this condition.

For a given tuple $y=\{(A_K,m_K)\}\in S(A)$, denote by $\tau(y)$ the set of pairs which appear in $y$ and which satisfy one of the following conditions:
\begin{enumerate}
\item $A_k^2\ne 0$ or
\item $c_1(A_k)\ne 0$.
\end{enumerate}

With this notation we can prove the following Lemma, analogous to Lemma 5.6, \cite{T4}:

\begin{lemma}

\begin{equation}
GT^V(A)(\mathcal I_A)=\sum_{y\in S(A)}\left[\prod_{\tau(y)}Per(y)\left[Ru^V(A_k,[\mathcal I_{A_k}])\right]^{m_k}\right] 
\end{equation}
\[
\times \left[\prod_{(A_k,m_k)\not\in\tau(y)}Qu^V(A_k,m_k)\right]
\]
where
\[
Per(y)=\frac{d_A!}{(d^k_1!d_2^k!)^{m_k}}\frac{l_A!}{(l^k_1!l^k_2!l^k_3!)^{m_k}}\frac{1}{m_k!}
\]

\end{lemma}

{\bf Remark:}  The sum in the above Lemma may incorporate submanifolds in the three exceptional cases allowing for components $V$.  This is the reason for defining $r(V,m)=0$, as these submanifolds will not contribute to the above sum.

\begin{proof}
This is a resummation of the defining sum
\[
GT^V(A)([\mathcal I_A])=\sum_{h\in\mathcal R_V(A,J,\mathcal I_A)}q(h).
\]
The key points are the following:
\begin{enumerate}
\item The permutation term $p(h)$ can be decomposed into a product of permutation terms $p(C_k,m_k,\mathcal I_k)$ stemming from the components $(C_k,m_k,\mathcal I_k)$ of $h$.  These are taken into account in the number $Ru^V(A_k,[\mathcal I_k])$ with the ordering on $[\mathcal I_k]$ induced by the ordering of $[\mathcal I_A]$.

\item Given $h=\{(C_k.m_k,\mathcal I_k)\}$ we can replace any $C_k$ by any other submanifold in $\mathcal K_V(J,\mathcal I_{A_k})$.  Hence all elements of class $A_k$ meeting the data $\mathcal I_k$ appear in the definition of the relative Gromov-Taubes invariant.  Hence resumming along homology classes simply reorders the sum but does not add any new terms.

\item The term $Per(y)$ accounts for permutations of the initial data on the components of $h$.  This is a purely combinatorial term, we have accounted for the topological effects of permuting the initial data in the term $Ru^V(A_k,[\mathcal I_k])$.

\end{enumerate}
\end{proof}

We have seen above, that the numbers $Ru^V(A,[\mathcal I_A])$ are invariants.  The proof of Theorem \ref{invG} will be complete with the results for multiply toroidal classes in Section \ref{taubesinv}, in particular Lemma \ref{A=V}, and the following Lemma:

\begin{prop}(Lemma 5.7, \cite{T4})\label{taubesprop}
Let $T\in H_2(X)$ be an indivisible toroidal class.  Let $n\ge 1$ be an integer.  Then $Qu(T,n)$ depends only on the deformation class of the symplectic form.
\end{prop}

\begin{example}Recall the example in genus 0, see Thm. \ref{spheres}.  In this case we determined the relative Ruan invariants for spheres in non-rational and non-ruled manifolds.  Assume that $A=\sum_{i=1}^nE_i$, each $E_i$ an exceptional sphere.  It then follows that
\begin{equation}
\mathcal R_V(A,J,\mathcal I_{A})=\left\{\begin{array}{cc} \{(E_1,1,\mathcal I_1),...,(E_m,1,\mathcal I_m)\}&  [E_i]\cdot \mathfrak V>0\;\forall i\\
\emptyset &  \mbox{otherwise}
\end{array}\right.
\end{equation}
which consists of one curve with $m$ components.  The invariant in the first case reduces to
\begin{equation}
GT^V(A)([\mathcal I_A])=q(\{(E_1,1,\mathcal I_1),...,(E_m,1,\mathcal I_m)\})=\prod_{i=1}^mRu^V(E_i,J,\mathcal I_i)=1.
\end{equation}
\end{example}

\end{document}